\documentclass[10pt, letterpaper]{article}
\usepackage[latin1]{inputenc}
\usepackage[T1]{fontenc}
\usepackage{amsmath,amssymb,amstext}
\usepackage{theorem}
\usepackage{multicol}
\usepackage[english]{babel}
\usepackage{datetime}
\usepackage{enumerate}
\usepackage{eufrak}
\usepackage{graphicx}
\usepackage{caption}
\usepackage{subcaption}
\usepackage{float}
\usepackage{multirow}
\usepackage{color}

\graphicspath{ {./images/} }

\newtheorem{theorem}{Theorem}

\newtheorem{corollary}[theorem]{Corollary}

\newtheorem{example}[theorem]{Example}

\newtheorem{lemma}[theorem]{Lemma}

\newtheorem{proposition}[theorem]{Proposition}
\newtheorem{remark}[theorem]{Remark}

\newenvironment{proof}[1][Proof]{\textbf{#1.} }{\ \rule{0.5em}{0.5em}}

\renewcommand{\geq}{\geqslant}
\def\leq{\leqslant}

\def\1{{\mathbf{1}}}

\def\1{{\mathbf{1}}}
\def\0.5{{\frac{1}{2}}}

\begin{document}

\title{\textbf{Asymptotics of Yule's nonsense correlation for
Ornstein-Uhlenbeck paths: a Wiener chaos approach.}}
\author{Soukaina Douissi \thanks{%
National School of Applied Sciences, Marrakech, Morocco. Email:\texttt{%
s.douissi@uca.ma}}, Frederi Viens \thanks{%
Department of Statistics and Probability, Michigan State University, USA.
E-mail: \texttt{viens@msu.edu}}, Khalifa Es-Sebaiy \thanks{%
Department of Mathematics, Faculty of Science, Kuwait University, Kuwait.
E-mail: \texttt{khalifa.essebaiy@ku.edu.kw}}.}
\date{\today}
\maketitle

\begin{abstract}
In this paper, we study the distribution of the so-called "Yule's nonsense 
correlation statistic" on a time interval $[0,T]$ for a time horizon $T>0$ ,
when $T$ is large, for a pair $(X_{1},X_{2})$ of independent 
Ornstein-Uhlenbeck processes. This statistic is by definition equal to :  
\begin{equation*}
\rho (T):=\frac{Y_{12}(T)}{\sqrt{Y_{11}(T)}\sqrt{Y_{22}(T)}},
\end{equation*}
where the random variables $Y_{ij}(T)$, $i,j=1,2$ are defined as  
\begin{equation*}
Y_{ij}(T):=\int_{0}^{T}X_{i}(u)X_{j}(u)du-T\bar{X}_{i}\bar{X_{j}},\text{ \ } 
\text{ \ }\bar{X}_{i}:=\frac{1}{T}\int_{0}^{T}X_{i}(u)du.
\end{equation*}
We assume $X_{1}$ and $X_{2}$ have the same drift parameter $\theta >0$. We 
also study the asymptotic law of a discrete-type version of $\rho (T)$, 
where $Y_{ij}(T)$ above are replaced by their Riemann-sum discretizations. 
In this case, conditions are provided for how the discretization (in-fill) 
step relates to the long horizon $T$. We establish identical normal 
asymptotics for standardized $\rho (T)$ and its discrete-data version. The 
asymptotic variance of $\rho (T)T^{1/2}$ is $\theta ^{-1}$. We also 
establish speeds of convergence in the Kolmogorov distance, which are of 
Berry-Esséen-type (constant*$T^{-1/2}$) except for a $\ln T$ factor. Our 
method is to use the properties of Wiener-chaos variables, since $\rho (T)$ 
and its discrete version are comprised of ratios involving three such 
variables in the 2nd Wiener chaos. This methodology accesses the Kolmogorov 
distance thanks to a relation which stems from the connection between the 
Malliavin calculus and Stein's method on Wiener space.
\end{abstract}

\section{Introduction}

In this paper, we study the normal asymptotics in law of the so-called
\textquotedblleft Yule's nonsense correlation statistic\textquotedblright\
on a time interval $[0,T]$ when the time horizon $T>0$ tends to infinity,
for two independent paths of the Ornstein-Uhlenbeck (OU) stochastic
processes. This statistic is defined as: 
\begin{equation}
\rho (T):=\frac{Y_{12}(T)}{\sqrt{Y_{11}(T)}\sqrt{Y_{22}(T)}},  \label{rho}
\end{equation}%
where the random variables $Y_{ij}(T)$, $i,j=1,2$ are given by 
\begin{equation}  \label{Y}
Y_{ij}(T):=\int_{0}^{T}X_{i}(u)X_{j}(u)du-T\bar{X}_{i}\left( T\right) \bar{
X_{j}}\left( T\right) ,\text{ \ }\text{ \ }\bar{X}_{i}\left( T\right) :=%
\frac{1}{T}\int_{0}^{T}X_{i}(u)du,
\end{equation}
and $(X_{1},X_{2})$ is a pair of two independent OU processes with the same
known drift parameter $\theta >0$, namely $X_{i}$ solves the linear SDE, for 
$i=1,2$ 
\begin{equation}
dX_{i}(t)=-\theta X_{i}(t)dt+dW^{i}(t),\text{ \ }t\geq 0  \label{OU}
\end{equation}%
with $X_{i}(0)=0$, $i=1,2$, where the driving noises $(W^{1}(t))_{t\geq 0}$, 
$(W^{2}(t))_{t\geq 0}$ are two independent standard Brownian motions (Wiener
processes). We also study the asymptotic law of a discrete-data version of $%
\rho (T)$, denote by $\tilde{\rho}(n)$ for $n$ observations, where the
Riemann integrals in (\ref{Y}) are replaced by Riemann-sum approximations.

It has been known since 1926 that a discrete version of the statistic $\rho $%
, which is the Pearson correlation coefficient, does not behave the same way
when the data from $X_{1}$ and $X_{2}$ are i.i.d. and as when they are the
discrete-time observations of a random walk. As is universally known for i.i.d. data, and also holds for shorter-memory models, $%
\tilde{\rho}(n)$ converges in probability to $0$ under all but the most
extreme circumstances (data coming from a distribution with no second
moment), but G. Udny Yule showed in \cite{Yule} that when the data come from
a random walk, $\tilde{\rho}(n)$ does not concentrate, and has a law which
seems to converge instead to a diffuse distribution on $(-1,1)$. The exact
variance and other statistical properties of this law remained unknown with
mathematical precision, though a 1986 paper \cite{phillips} by P.C.B.
Phillips showed that the limiting law of $\tilde{\rho}(n)$ with simple
symmetric random-walk data rescaled to the time interval $[0,1]$ is the same
as the law of $\rho (1)$ for two independent Wiener processes, which is
indeed necessarily diffuse and fully supported on $(-1,1)$. This advance
prompted several talented prominent probabilists to look for ways of
computing statistics of $\rho (1)$, if even only its variance, but this
remained elusive until 2017, when Ph. Ernst and two collaborators (one
posthumous) provided a closed-form expression for $Var[\rho (1)]$ in \cite%
{ESW}. Since then, other advances on the moments of $\rho (1)$ have been
made, particularly \cite{ERQ}, and recent progress was recorded in \cite{EHV}
on how to compute the momens of $\tilde{\rho}(n)$ when the paths $%
(X_{1},X_{2})$ are independent Gaussian simple-symmetric random walks. In
all cases mentioned in this paragraph, the asymptotic behavior of $\tilde{
\rho}(n)$ in law (scaled appropriately in time) is necessarily that of $\rho
(1)$ for two independent Wiener processes.

This leaves open the question of what happens to $\tilde{\rho}(n)$ when the
paths $(X_{1},X_{2})$ deviate substantially from Wiener paths or random
walks. Wiener (resp. random walk) paths have the property of exact (resp.
approximate) self-similarily. We take up the question of using different
kinds of paths, with the simplest possible example of a clear alternative to
self-similar processes, namely the ubiquitous mean-reverting OU processes.
The property of mean reversion is so distinct from self-similarily, that the
behavior of $\rho (T)$ changes drastically from one to the other. Note that
these two classes of processes are simply those satisfying (\ref{OU}) with $%
\theta =0$ (Wiener process) or $\theta \neq 0$ (OU\ process). To illustrate
the point of how distinct these processes are, let us extend the scope of
this paper momentarily, to include all processes defined by (\ref{OU}) (with
or without $\theta =0$), by replacing $W$ with a fractional Brownian
motion (fBm) denoted by $B^{H}$, for some $H\in (0,1)$. Like the Wiener
process, which corresponds to $H=1/2$, the self-similar property of $B^{H}$
simply states that for any fixed real constant $a$, $B^{H}(a\cdot
)=a^{H}B^{H}\left( \cdot \right) $ in law. By using this property with $a=T$
via the change of variable $u^{\prime }=u/T$ in the Riemann integrals
defining $\rho (T)$, we obtain immediately the equality in law 
\begin{eqnarray*}
Y_{ij}\left( T\right)  &=&\int_{0}^{1}T^{H}X_{i}(u^{\prime
})T^{H}X_{j}(u^{\prime })Tdu^{\prime }-T^{-1}T^{H+1}\bar{X}_{i}\left(
1\right) T^{H+1}\bar{X_{j}}\left( 1\right)  \\
&=&T^{2H+1}Y_{ij}\left( 1\right) 
\end{eqnarray*}%
and therefore%
\begin{equation*}
\mathcal{L}\left( \rho \left( T\right) \right) =\mathcal{L}\left( \rho
\left( 1\right) \right) .
\end{equation*}%
In fact, we only used the property of self-similarity to get the above. In
other words, for any pair of self-similar processes, the law of the nonsense
correlation $\rho \left( T\right) $ is constant as the time horizon $T$
increases. In stark contrast, in this paper, we show that, for a pair of OU
processes, the law of $\rho \left( T\right) $ converges to the Dirac mass at
0. As mentioned, we show more: a central limit theorem for $\rho \left(
T\right) \sqrt{T}$ (the mean of $\rho \left( T\right) $ is always 0), with
asymptotic variance equal to $\theta ^{-1}$, and a speed of convergence of $%
\mathcal{L}\left( \rho \left( T\right) T^{1/2}\right) $ to $N\left( 0,\theta
^{-1}\right) $ in Kolmogorov metric at the rate $T^{-1/2}\ln T$.

The discrete-observation part of this paper simply replaces the Riemann
integrals by the Riemann sums, for instance replacing the first integral in $%
Y_{ij}\left( T\right) $ by $\Delta
_{n}\sum\limits_{k=0}^{n-1}X_{i}(t_{k})X_{j}(t_{k})$ where $\Delta _{n}=T/n$
and $t_{k}=kT/n$.\ We denote the resulting empirical correlation by $\tilde{%
\rho}(n)$ rather than $\rho (T)$. It is convenient to note that $T=n\Delta
_{n}$ can be thought of as depending on $n$, and we will systematically
emphasize this by denoting $T=T_{n}$. It is assumed that the discretization
step $\Delta _{n}$ converges to $0$ while $T_{n}=T$ tends to $\infty $,
which means that $n\gg T_{n}$ in our asymptotics. We provide a full range of speeds of convergence in central limit theorem depending on how fast $\Delta _{n}$ converges to $0.$
We find in fact that we must have $T_{n}\Delta _{n}=n\Delta
_{n}^{2}\rightarrow 0$ as $n\rightarrow \infty ,$ and we also note that $%
n\Delta _{n}=T_{n}=T\rightarrow \infty $, as well it should. Our convergence
result, which immediately implies the central limit theorem $%
\lim_{n\rightarrow \infty }\mathcal{L}\left( \tilde{\rho}(n)T_{n}^{1/2}%
\right) =\mathcal{N}\left( 0,\theta ^{-1}\right) $, is%
\begin{eqnarray*}
d_{Kol}\left( \sqrt{\theta }\sqrt{T_{n}}\tilde{\rho}(n),\mathcal{N}%
(0,1)\right)  &\leqslant &c(\theta )\times \ln (n\Delta _{n})\max \left(
(n\Delta _{n})^{-1/2},(n\Delta _{n}^{2})^{\frac{1}{3}}\right)  \\
&=&c(\theta )\times \ln (T_{n})\max \left( T_{n}^{-1/2},(T_{n}\Delta _{n})^{%
\frac{1}{3}}\right) .
\end{eqnarray*}%
From this, we can immediately read off that a rather optimal rate of
sampling of our discrete data is one for which the two terms in the max are
of the same order, i.e. $T_{n}^{-1/2}\asymp (T_{n}\Delta _{n})^{\frac{1}{3}}$%
, which is equivalent to requiring that $\Delta _{n}$ be of order $%
T_{n}^{-5/2}$, which in turn, since $T_{n}=n\Delta _{n}$, is equivalent to $%
\Delta _{n}$ of order $n^{-5/7}$. This is explained in more detail in the
conclusion of the section on discrete data. In any case, in Kolmogorov
distance, we see that the best rate of convergence of $\sqrt{T_{n}}\tilde{%
\rho}(n)$ to $N\left( 0,\theta ^{-1}\right) $ is of order $T_{n}^{-1/2}\ln
T_{n}$, which is exactly the same rate as in the case of continuous data,
and which occurs for a relatively frequency of observations of order $\Delta
_{n}^{-1}=T_{n}^{5/2}$ over unit intervals. Lower frequency of observations
lead to slower convergence rate in Kolmogorov distance in the scale of the
time horizon $T_{n}$ compared to continuous observation. Higher frequency of
observation leads to the same rate $T_{n}^{-1/2}\ln T_{n}$ as with
continuous observations, but this can be considered wasteful since the same
rate was achieved at the optimal frequency of $\Delta _{n}^{-1}=T_{n}^{5/2}$
per unit time. It is worth stating again that these results in discrete
time, pertaining to the convergence rate in the CLT for $\sqrt{T_{n}}\tilde{%
\rho}(n)$, are a second-order result compared to the CLT itself, i.e. 
\begin{equation*}
\lim_{n\rightarrow \infty }\sqrt{T_{n}}\tilde{\rho}(n)=\lim_{T\rightarrow
\infty }\sqrt{T}\rho (T)=\mathcal{N}(0,\theta ^{-1})
\end{equation*}%
which holds in law identically in both the discrete and continuous data
cases.

As mentioned, we use techniques from analysis on Wiener chaos to prove the
above results. Some of these results are technical and novel, and we provide
here a few points in the hopes of enlightening the methods. A key element
comes from the connection discovered by I. Nourdin and G. Peccati (see \cite%
{NP-book}) between the Malliavin calculus and Stein's method. In that
connection, the distance in law between a random variable $X$ and the
standard normal law can be measured to some extent by comparing the
Hilbert-space norm of the Malliavin derivative $DX$ to the value $1$, which
is the value one would find for the norm of the Malliavin derivative of a
standard normal variable $N$ under any reasonable coupling of $X$ and $N,$
i.e. under any reasonable representation of $X$ on Wiener space. The
question of how to represent $X$ on Wiener space is typically trivial when
dealing with functionals of stochastic processes based on Wiener processes,
and this is certainly the case in our paper. The question of whether $DX$ is
an adeqate functional of $X$ to make the comparison with $\mathcal{N}\left(
0,1\right) $ is less trivial. The original work in \cite{NP2009b} noted that
it is sufficient for variables on Wiener chaos, and used an auxiliary random
variable $G_{X}$ which is slightly more involved than $\left\Vert
DX\right\Vert ^{2}$ to establish broader convergence in law beyond fixed
chaos. That same random variable $G_{X}$ was used in \cite{nv} to
characterize laws on Wiener chaos at the level of densities, and was used
specifically in \cite[Theorem 2.4]{NP2009} to measure distances between laws
in the Kolmogorov metric. We use that theorem herein, by applying it
separately to all three Wiener chaos components $Y_{11,}Y_{22},Y_{12}$ which
are used to calculate $\rho \left( T\right) $, noting as in \cite{NP2009b}
that $G_{X}=2\left\Vert DX\right\Vert ^{2}$.

The above elements are explained in the preliminary section on analysis on
Wiener space below. They are used herein via standard computations of
variances and differentiation and product rules for variables on Wiener
chaos which are represented as double Wiener integrals with respect to the $%
W_{i}$'s, leading to computing the asymptotic variance of the rescaled
numerator $T^{-1/2}Y_{12}(T)$, namely $1/4\theta ^{3}$, and the speed of
convergence of the variances to this limit, including precise estimations of
the constants in this rate of convergence as functions of $\theta $. It
turns out that the rescaled denominator $T^{-1}(Y_{11}\left( T\right)
Y_{22}\left( T\right) )^{1/2}$ does not have normal fluctuations, but rather
converges to the constant $1/2\theta $. We establish this too. Adding to
this that the numerator, as a second-chaos variable, has mean zero, this
indicates that the entire rescaled fraction $\rho $ should converge in law
to $\mathcal{N}\left( 0,\theta ^{-1}\right) .$ Finding a presumably sharp
rate of this convergence in Kolmogorov metric is the main technical issue we
tackle in this paper. Establishing this for the numerator alone is a key
quantitative estimate. We use \cite[Theorem 2.4]{NP2009} and our ability to
compute the norm of the Malliavin derivative of the first double Wiener
integral $\int_{0}^{T}X_{i}(u)X_{j}(u)du$ in the expression for $%
Y_{12}\left( T\right) ,$ and we find a rate of normal convergence of order $%
T^{-1/2}$. However, we must also handle the second term in $Y_{12}\left(
T\right) $, which is the (rescaled) product $\bar{X_{1}}(T)\bar{X_{2}}(T)$
of two independent normal variables, which are both non-independent from the
first part of $Y_{12}\left( T\right) $. For this, we appeal to a 1971
theorem of Michel and Pfanzagl \cite{MP} which allows us to decouple the
dependence of a sum (resp. a ratio) of two variables when comparing them to
a normal law in Kolmogorov distance. We specialize this theorem to the case
when the second summand (resp. the denominator) is a product normal variable
(resp. the root of a product normal), establishing an optimal use of it in
this special case. See Proposition \ref{estim-kol}, Corollary \ref{kol-sum},
Proposition \ref{estim-num} and estimate (\ref{rhoestim-1}), and estimate (%
\ref{majorY11}).

This optimal use of this decoupling technique comes at the very small cost
of adding a factor of $\ln T$ to our rate of convergence. We believe this
factor is optimal given our use of \cite{MP}, and is determined by the
weight of the tail of a product normal law, which is asymptotically the same
as the tail of a chi-square variable with one degree of freedom, which in
logarithmic scale, is the same as an exponential tail . The interested
reader can check that any use of Holder's inequality or similar methods
based on moments, cannot achieve this more efficient method, leading instead
to a rate of convergence of $T^{-\alpha }$ for $\alpha <1/2$. The reader
will also observe our use of the fine structure of the second Wiener chaos
as a separable Hilbert space, to deal with the tail distribution of $\rho $%
's denominator terms. This structure is documented for instance in \cite[%
Section 2.7.4]{NP-book} where it is shown that every second-chaos variable
can be represented as a series $\sum_{k}\lambda _{k}\chi _{k}^{2}$ where $%
\left( \chi _{k}^{2}\right) _{k}$ is a sequence of i.i.d. mean-zero
chi-square variables with one degree of freedom, and $\left( \lambda
_{k}\right) _{k}$ is in $\ell ^{2}$. In our case, the reader will observe
that the terms in the denominator of $\rho $ also contain non-zero
expectations, that their $\lambda _{k}$'s are positive and in $\ell ^{1}$,
and that the expectations equal $\sum_{k}\lambda _{k}$. This fact is
essential to us being able to control the denominator.

The techniques used to establish results in the case of discrete
observations are similar to those in the continuous case. Additional
ingredients include the rate of convergence of the Riemann-sum version of
the first integral in $Y_{12}\left( T_{n}\right) $ to its limit. This rate
turns out to be $n\Delta _{n}^{2}$ where, as mentioned, $n$ is the number of
observations in $(0,T_{n}],$ and the regular mesh is $\Delta _{n}$. The use
of the aforementioned Michel-Pfanzagl theorem from \cite{MP} to deal with
the product-normal term in the numerator has to be optimized against this
dicretization error; this is where the term $\left( n\Delta _{n}^{2}\right)
^{1/3}$ comes from, whereas the term $\left( n\Delta _{n}\right) ^{-1/2}$ is
none other than the same convergence rate $T^{-1/2}$ for the numerator in
Kolmogorov distance as in the continuous case. See Lemma \ref{norm-delta}
and Proposition \ref{An-kol}. The use of the sum version of the
Michel-Pfanzagl theorem from our Corollary \ref{kol-sum} leads again to a
leading log correction factor $\ln T_{n}=\ln \left( n\Delta _{n}\right) $.
The denominator terms also require a careful analysis, though no additional
ideas are needed beyond what was already established in the continuous case.

With this roadmap summary complete, the structure of this paper should
appear as straightforward. We begin with a section of preliminaries
presenting the tools needed from analysis on Wiener space, followed by a
section covering the convergence in the continuous case, and then a section
dealing with the case of discrete data. The final section provides some
numerics to illustrate the convergence rates in practice, wherein we find
that in discrete time, the time-scaled $\tilde{\rho}\left( n\right) $ does
indeed behave in distribution largely like a normal with variance $\theta
^{-1}$, even without using the optimal observation frequency.

\section{Preliminaries}

\subsection{Elements of Analysis on Wiener space\label{Elements}}

With $\left( \Omega ,\mathcal{F},\mathbf{P}\right) $ denoting the Wiener
space of a standard Wiener process $W$, for a deterministic function $h\in
L^{2}\left( \mathbf{R}_{+}\right) =:{{\mathcal{H}}}$, the Wiener integral $%
\int_{\mathbf{R}_{+}}h\left( s\right) dW\left( s\right) $ is also denoted by 
$W\left( h\right) $. The inner product $\int_{\mathbf{R}_{+}}f\left(
s\right) g\left( s\right) ds$ will be denoted by $\left\langle
f,g\right\rangle _{{\mathcal{H}}}$.

\begin{itemize}
\item \textbf{The Wiener chaos expansion}. For every $q\geq 1$, ${\mathcal{H}%
}_{q}$ denotes the $q$th Wiener chaos of $W$, defined as the closed linear
subspace of $L^{2}(\Omega )$ generated by the random variables $%
\{H_{q}(W(h)),h\in {{\mathcal{H}}},\Vert h\Vert _{{\mathcal{H}}}=1\}$ where $%
H_{q}$ is the $q$th Hermite polynomial. Wiener chaos of different orders are
orthogonal in $L^{2}\left( \Omega \right) $. The so-called Wiener chaos
expansion is the fact that any $X\in L^{2}\left( \Omega \right) $ can be
written as 
\begin{equation}
X=\mathbf{E}[X]+\sum_{q=1}^{\infty }X_{q}  \label{WienerChaos}
\end{equation}%
for some $X_{q}\in {\mathcal{H}}_{q}$ for every $q\geq 1$. This is
summarized in the direct-orthogonal-sum notation $L^{2}\left( \Omega \right)
=\oplus _{q=0}^{\infty }{\mathcal{H}}_{q}$. Here ${\mathcal{H}}_{0}$ denotes
the constants.

\item \textbf{Relation with Hermite polynomials. Multiple Wiener integrals}.
The mapping ${I_{q}(h^{\otimes q}):}=q!H_{q}(W(h))$ is a linear isometry 
between the symmetric tensor product ${\mathcal{H}}^{\odot q}$ (equipped 
with the modified norm $\Vert .\Vert _{{\mathcal{H}}^{\odot q}}=\sqrt{q!}
\Vert .\Vert _{{\mathcal{H}}^{\otimes q}}$) and ${\mathcal{H}}_{q}$. Hence,
for $X$ and its Wiener chaos expansion (\ref{WienerChaos})  above, each term 
$X_{q}$ can be interpreted as a multiple Wiener integral $ I_{q}\left(
f_{q}\right) $ for some $f_{q}\in {\mathcal{H}}^{\odot q}$.

\item \textbf{Isometry Property-Product formula}. For any integers $%
1\leqslant q\leqslant p$ and $f\in \mathcal{H}^{\odot p}$ and $g\in \mathcal{%
H}^{\odot q}$, we have 
\begin{eqnarray}  \label{Isometryproperty01}
\mathbf{E}[I_{p}(f)I_{q}(g)] = &&\left\{ 
\begin{array}{ll}
p! \langle f , g \rangle _{{\mathcal{H}^{\otimes p}}} & \mbox{ if } p=q \\ 
~~ &  \\ 
0 & \mbox{ } \text{otherwise}.%
\end{array}%
\right.
\end{eqnarray}
For any integers $p$, $q \geq 1$ and symmetric integrands $f\in \mathcal{H}%
^{\odot p}$ and $g\in \mathcal{H}^{\odot q}$, 
\begin{equation}
I_{p}(f)I_{q}(g)=\sum_{r=0}^{p\wedge q}r!{C}_{p}^{r}{C}_{q}^{r}I_{p+q-2r}(f%
\tilde{\otimes _{r}}g);  \label{productformula01}
\end{equation}%
where $f\otimes _{r}g$ is the contraction of order $r$ of $f$ and $g$ which
is an element of ${\mathcal{H}}^{\otimes (p+q-2r)}$ defined by  
\begin{eqnarray*}
&&(f\otimes _{r}g)(s_{1},\ldots ,s_{p-r},t_{1},\ldots ,t_{q-r}) \\
&& =\int_{\mathbf{R}_{+}^{p+q-2r}}f(s_{1},\ldots ,s_{p-r},u_{1},\ldots
,u_{r})g(t_{1},\ldots ,t_{q-r},u_{1},\ldots ,u_{r})\,du_{1}\cdots du_{r}.
\end{eqnarray*}

while $(f\tilde{\otimes _{r}}g)$ denotes its symmetrization. More generally
the symmetrization $\tilde{f}$ of a function $f$ is defined by $\tilde{f}%
(x_{1},...,x_{p}) = \frac{1}{p!} \sum\limits_{\sigma}
f(x_{\sigma(1)},...,x_{\sigma(p)})$ where the sum runs over all permutations 
$\sigma$ of $\{1,...,p\}$.  The special case for $p=q=1$ in (\ref%
{productformula01}) is particularly handy, and can be written in  its
symmetrized form: 
\begin{equation}  \label{produit}
I_{1}(f)I_{1}(g)=2^{-1}I_{2}\left( f\otimes g+g\otimes f\right) +\langle
f,g\rangle _{{\mathcal{H}}}.
\end{equation}
where $f\otimes g$ means the tensor product of $f$ and $g$.

\item \textbf{Hypercontractivity in Wiener chaos}. For $h\in {\mathcal{H}}%
^{\otimes q}$, the multiple Wiener integrals $I_{q}(h)$, which exhaust the
set ${\mathcal{H}}_{q}$, satisfy a hypercontractivity property (equivalence
in ${\mathcal{H}}_{q}$ of all $L^{p}$ norms for all $p\geq 2$), which
implies that for any $F\in \oplus _{l=1}^{q}{\mathcal{H}}_{l}$ (i.e. in a
fixed sum of Wiener chaoses), we have 
\begin{equation}
\left( \mathbf{E}\big[|F|^{p}\big]\right) ^{1/p}\leqslant c_{p,q}\left( 
\mathbf{E}\big[|F|^{2}\big]\right) ^{1/2}\ \mbox{ for any }p\geq 2.
\label{hypercontractivity}
\end{equation}%
The constants $c_{p,q}$ above are known with some precision when $F\in {%
\mathcal{H}}_{q}$: by Corollary 2.8.14 in \cite{NP-book}, $c_{p,q}=\left(
p-1\right) ^{q/2}$. 

\item \textbf{Malliavin derivative}. For any function $\Phi \in C^{1}\left( 
\mathbf{R}\right) $ with bounded derivative, and any $h\in {\mathcal{H}}$,
the Malliavin derivative $D$ of the random variable $X:=\Phi \left( W\left(
h\right) \right) $ is defined to be consistent with the following chain
rule: 
\begin{equation*}
DX:X\mapsto D_{r}X:=\Phi ^{\prime }\left( W\left( h\right) \right) h\left(
r\right) \in L^{2}\left( \Omega \times \mathbf{R}_{+}\right) .
\end{equation*}%
A similar chain rule holds for multivariate $\Phi $. One then extends $D$ to
the so-called Gross-Sobolev subset $\mathbf{D}^{1,2}\varsubsetneqq
L^{2}\left( \Omega \right) $ by closing $D$ inside $L^{2}\left( \Omega
\right) $ under the norm defined by $\left\Vert X\right\Vert _{1,2}^{2}:=%
\mathbf{E}\left[ X^{2}\right] +\int_{\mathbf{R}_{+}}\mathbf{E|}D_{r}X|^{2}dr.
$ All Wiener chaos random variable are in the domain $\mathbf{D}^{1,2}$ of $D
$ . In fact this domain can be expressed explicitly for any $X$ as in (\ref%
{WienerChaos}): $X\in \mathbf{D}^{1,2}$ if and only if $\sum_{q}qq!\Vert
f_{q}\Vert _{{\mathcal{H}}^{\otimes q}}^{2}<\infty $. 

\item \textbf{Generator }$L$ \textbf{of the Ornstein-Uhlenbeck semigroup}.
The linear operator $L$ is defined as being diagonal under the Wiener chaos
expansion of $L^{2}\left( \Omega \right) $: ${\mathcal{H}}_{q}$ is the
eigenspace of $L$ with eigenvalue $-q$, i.e. for any $X\in {\mathcal{H}}_{q}$
, $LX=-qX$. We have $Ker$($L)=$ ${\mathcal{H}}_{0}$, the constants. The
operator $-L^{-1}$ is the negative pseudo-inverse of $L$, so that for any $%
X\in {\mathcal{H}}_{q}$, $-L^{-1}X=q^{-1}X$.

\item \textbf{Kolmogorov distance.} Recall that, if $X,Y$ are two
real-valued random variables, then the Kolmogorov distance between the law
of $X$ and the law of $Y$ is given by 
\begin{equation*}
d_{Kol}\left( X,Y\right) =\sup_{z\in \mathbf{R}}\left\vert \mathbf{P}\left[
X\leqslant z\right] -\mathbf{P}\left[ Y\leqslant z\right] \right\vert 
\end{equation*}%
If $X\in \mathbb{D}^{1,2}$, with $\mathbf{E}[X]=0$ and $Y=\mathcal{N}(0,1)$,
then (Theorem 2.4 in \cite{NP2009}), then 
\begin{equation*}
d_{Kol}\left( X,Y\right) \leqslant \sqrt{\mathbf{E}[(1-\langle
DX,-DL^{-1}X\rangle _{\mathcal{H}})^{2}]}
\end{equation*}%
If moreover, $X=I_{q}(f)$ for some $q\geq 2$, $f\in \mathcal{H}^{\odot q}$,
then $\langle DX,-DL^{-1}X\rangle _{\mathcal{H}}=q^{-1}\Vert DX\Vert _{%
\mathcal{H}}^{2}$, and thus in this case

\begin{equation}
d_{Kol}\left( X,Y\right) \leqslant \sqrt{\mathbf{E}[(1-q^{-1}\Vert DX\Vert _{%
\mathcal{H}}^{2})^{2}]}  \label{est-kol}
\end{equation}
\end{itemize}

\begin{lemma}
\label{Borel-Cantelli} Let $\gamma >0$. Let $(Z_{n})_{n\in \mathbb{N}}$ be a
sequence of random variables. If for every $p\geq 1$ there exists a constant
$c_{p}>0$ such that for all $n\in \mathbb{N}$,  
\begin{equation*}
\Vert Z_{n}\Vert _{L^{p}(\Omega )}\leqslant c_{p}\cdot n^{-\gamma },
\end{equation*}
then for all $\varepsilon >0$ there exists a random variable $\eta 
_{\varepsilon }$ which is almost surely finite such that  
\begin{equation*}
|Z_{n}|\leqslant \eta _{\varepsilon }\cdot n^{-\gamma +\varepsilon }\quad %
\mbox{almost surely}
\end{equation*}
for all $n\in \mathbb{N}$. Moreover, $E|\eta _{\varepsilon }|^{p}<\infty $ 
for all $p\geq 1$.
\end{lemma}

\section{Continuous observations}

In this section, we compute the asymptotic variance of $\rho (T)$ and its
normal fluctuations for large $T,$ by working with each of the three terms
which appear in its definition. For the sake of convienence and compactness
of notation, we construct a two-sided Brownian motion $(W(t))_{t\in \mathbb{R%
}}$ from the two independent Brownian motions $(W^{1}(t))_{t\geq 0}$ and $%
(W^{2}(t))_{t\geq 0}$ as follows : 
\begin{equation*}
W(t):=W^{1}(t)\mathbf{1}_{\{t\geq 0\}}+W^{2}(-t)\mathbf{1}_{\{t<0\}},\text{
\ }t\in \mathbb{R.}
\end{equation*}%
The following lemma will be convenient in the sequel.

\begin{lemma}
\label{I2}  Let $f$, $g$ $\in L^{2}(\mathbb{R}_{+})$, then  
\begin{equation*}
I_{1}^{W^{1}}(f)I_{1}^{W^{2}}(g) = I_{2}^{W}(\bar{f} \otimes \bar{\bar{g}})
\end{equation*}
where $\bar{f}$, $\bar{\bar{g}}$ in $L^{2}(\mathbb{R})$ are defined by  
\begin{equation*}
\bar{f}(x) = f(x) \mathbf{1}_{\{x \geq 0\}}, \text{ \ } \bar{\bar{g}}(x) =
-g(-x) \mathbf{1}_{\{x < 0\}}.
\end{equation*}
\end{lemma}
\begin{proof}
Using the product formula of multiple integrals, we have 
\begin{eqnarray*}
I_{2}^{W}(\bar{f} \otimes \bar{\bar{g}})&=&I_{1}^{W}(\bar{f})I_{1}^{W}(\bar{%
\bar{g}})-E\left[I_{1}^{W}(\bar{f})I_{1}^{W}(\bar{\bar{g}})\right] \\
&=&\left(\int_{\mathbb{R}}\bar{f}(x)dW_x\right)\left(\int_{\mathbb{R}}\bar{%
\bar{g}}(x)dW_x\right)-E\left<\bar{f},\bar{\bar{g}}\right>_{L^{2}(\mathbb{R}%
)} \\
&=&\left(\int_{0}^{\infty}f(x)dW^1_x\right)\left(\int_{0}^{\infty}g(x)dW^2_x%
\right) \\
&=&I_{1}^{W^{1}}(f)I_{1}^{W^{2}}(g),
\end{eqnarray*}%
which completes the proof.
\end{proof}

\subsection{Asymptotic distribution of $\frac{Y_{12}(T)}{\protect\sqrt{T}}$ :%
}

The numerator of $\rho (T)$ can be written as follows 
\begin{equation}
\frac{Y_{12}(T)}{\sqrt{T}}=F_{T}-\sqrt{T}\bar{X_{1}}(T)\bar{X_{2}}(T)
\label{num}
\end{equation}%
where $F_{T}:=\frac{1}{\sqrt{T}}\int_{0}^{T}X_{1}(t)X_{2}(t)dt$. Using the
notation $I_{1}^{W}$ for the Wiener integral with respect to $W$, since $%
X_{i}(t)=\int_{0}^{t}e^{-\theta (t-u)}dW^{i}(u)$ = $I_{1}^{W^{i}}(f_{t})$, $%
i=1,2$ where $f_{t}(.):=e^{-\theta (t-.)}\mathbf{1}_{[0,t]}(.)$ we can write
using Lemma \ref{I2} 
\begin{align}
F_{T}& =\frac{1}{\sqrt{T}}%
\int_{0}^{T}I_{1}^{W^{1}}(f_{t})I_{1}^{W^{2}}(f_{t})dt  \label{FT} \\
& =I_{2}^{W}(h_{T}),  \notag
\end{align}%
with $h_{T}\in L^{2}([-T,T]^{2})$ is given 
\begin{align}
h_{T}:[-T,T]^{2}& \rightarrow \mathbb{R}  \notag \\
(x,y)& \mapsto \frac{1}{\sqrt{T}}\int_{0}^{T}\bar{f}_{t}(x)\bar{\bar{f}}%
_{t}(y)dt  \label{noyau-hT}
\end{align}%
On the other hand, we have 
\begin{align*}
h_{T}(x,y)& =\frac{1}{\sqrt{T}}\int_{0}^{T}-e^{-2\theta t}e^{\theta
x}e^{-\theta y}\mathbf{1}_{[0,t]}(x)\mathbf{1}_{[-t,0]}(y)dt \\
& =\frac{1}{\sqrt{T}}\int_{0}^{T}-e^{-2\theta t}e^{\theta x}e^{-\theta y}%
\mathbf{1}_{[x\vee -y,T]}(t)\mathbf{1}_{[0,T]}(x)\mathbf{1}_{[-T,0]}(y)dt \\
& =\frac{1}{2\theta }\frac{1}{\sqrt{T}}e^{\theta x}e^{-\theta y}\left[
e^{-2\theta T}-e^{-2\theta (x\vee -y)}\right] \mathbf{1}_{[0,T]}(x)\mathbf{1}%
_{[-T,0]}(y).
\end{align*}%
Note that the kernel $h_{T}$ is not symmetric, in the sequel we will denote $%
\tilde{h}_{T}$ its systematization defined by $\tilde{h}_{T}(x,y):=\frac{1}{2%
}(h_{T}(x,y)+h_{T}(y,x))$. We are now ready to compute the asymptotic
variance of the main term in the numerator of $\rho $.

\begin{lemma}
\label{varianceFT} With $F_{T}$ defined in (\ref{FT}), then 
\begin{equation*}
\left\vert \mathbf{E}[F_{T}^{2}]-\frac{1}{4\theta ^{3}}\right\vert \leqslant 
\frac{C(\theta )}{T},
\end{equation*}%
where $C(\theta ):=\frac{7+8\theta }{16\theta ^{4}}$. In particular, $%
\lim_{T\rightarrow \infty }\mathbf{E}[F_{T}^{2}]=\frac{1}{4\theta ^{3}}$.
\end{lemma}
\begin{proof}
We have 
\begin{align*}
\mathbf{E}[F_{T}^{2}]& =\mathbf{E}[I_{2}^{W}(h_{T})^{2}] \\
& =2\times \Vert \tilde{h}_{T}\Vert _{L^{2}([-T,T]^{2})}^{2} \\
& =\frac{1}{T}\frac{1}{4\theta ^{2}}\int_{-T}^{0}\int_{0}^{T}e^{2\theta
x}e^{-2\theta y}\left[ e^{-2\theta T}-e^{-2\theta (x\vee -y)}\right] ^{2}dxdy
\\
& =\frac{1}{T}\frac{1}{4\theta ^{2}}\int_{0}^{T}\int_{0}^{T}e^{2\theta
x}e^{2\theta z}\left[ e^{-2\theta T}-e^{-2\theta (x\vee z)}\right] ^{2}dxdz
\\
& =\frac{1}{T}\frac{1}{2\theta ^{2}}\int_{0}^{T}\int_{0}^{z}e^{2\theta
x}e^{2\theta z}\left[ e^{-2\theta T}-e^{-2\theta z}\right] ^{2}dxdz \\
& =\frac{1}{T}\frac{1}{4\theta ^{3}}\left[ \int_{0}^{T}(e^{-2\theta
(T-y)}-1)^{2}dy-\int_{0}^{T}e^{-2\theta y}(e^{-2\theta (T-y)}-1)^{2}dy\right]
\\
& =:A_{1}(T)+A_{2}\left( T\right) 
\end{align*}%
where%
\begin{align*}
|A_{1}(T)-\frac{1}{4\theta ^{3}}|& :=\left\vert \frac{1}{T}\frac{1}{4\theta
^{3}}\int_{0}^{T}(e^{-2\theta (T-y)}-1)^{2}dy-\frac{1}{4\theta ^{3}}%
\right\vert  \\
& =\left\vert \frac{1}{T}\frac{1}{4\theta ^{3}}[\frac{1}{4\theta }%
(1-e^{-4\theta T})+\frac{1}{\theta }(e^{-2\theta T}-1)+T]-\frac{1}{4\theta
^{3}}\right\vert  \\
& \leqslant \frac{5}{16\theta ^{4}}\times \frac{1}{T},
\end{align*}%
and%
\begin{align*}
|A_{2}(T)|& :=\left\vert \frac{1}{T}\frac{1}{4\theta ^{3}}%
\int_{0}^{T}e^{-2\theta y}(e^{-2\theta (T-y)}-1)^{2}dy\right\vert  \\
& \leqslant \frac{1}{T}\frac{1}{8\theta ^{4}}(1-e^{-4\theta T})+\frac{1}{%
2\theta ^{3}}e^{-2\theta T} \\
& \leqslant \frac{(1+4\theta )}{8\theta ^{4}}\times \frac{1}{T}.
\end{align*}
\end{proof}

\begin{proposition}
\label{cvg-loi-FT} Let $F^\theta_T := 2 \theta^{3/2} F_T$ and $N \sim 
\mathcal{N}(0,1)$, then we have 
\begin{equation*}
d_{Kol}(F^\theta_{T},N) \leqslant \frac{c(\theta)}{\sqrt{T}}.
\end{equation*}
where $c(\theta) := \sqrt{(2 + \frac{7}{4 \theta})^2 + \frac{3^3}{4 \theta}}$%
. Consequently $F_{T} \overset{\mathcal{L}}{\longrightarrow} \mathcal{N}%
\left(0, \frac{1}{4 \theta^3}\right)$  as $T \rightarrow + \infty$.
\end{proposition}
\begin{proof}
We will use the estimate (\ref{est-kol}) recalled in the preliminaries in
order to prove this proposition. We have $D_{t}F_{T}^{\theta }=4\theta
^{3/2}I_{1}^{W}(\tilde{h}_{T}(.,t))$, $t\in \lbrack -T,T]$, hence 
\begin{align*}
\frac{1}{2}\Vert DF_{T}^{\theta }\Vert _{L^{2}([-T,T])}^{2}& =\frac{1}{2}%
\int_{-T}^{T}(D_{t}F_{T}^{\theta })^{2}dt \\
& =8\theta ^{3}\int_{-T}^{T}I_{1}^{W}(\tilde{h}_{T}(.,t))^{2}dt \\
& =8\theta ^{3}\left[ \int_{-T}^{T}I_{2}^{W}(\tilde{h}_{T}(.,t)\otimes 
\tilde{h}_{T}(.,t))dt+\int_{-T}^{T}\Vert \tilde{h}_{T}(.,t)\Vert
_{L^{2}([-T,T])}^{2}\right]  \\
& =8\theta ^{3}\left[ I_{2}^{W}(\tilde{h}_{T}\otimes _{1}\tilde{h}%
_{T})+\Vert \tilde{h}_{T}\Vert _{L^{2}([-T,T]^{2})}^{2}\right] ,
\end{align*}%
where we used the product formula (\ref{produit}) and the fact that the
kernel $\tilde{h}_{T}\otimes _{1}\tilde{h}_{T}$ is symmetric. Thus 
\begin{equation}
\mathbf{E}\left[ (1-\frac{1}{2}\Vert DF_{T}^{\theta }\Vert
_{L^{2}([-T,T])}^{2})^{2}\right] =(\mathbf{E}[(F_{T}^{\theta
})^{2}]-1)^{2}+2^{7}\theta ^{6}\times \Vert \tilde{h}_{T}\otimes _{1}\tilde{h%
}_{T}\Vert _{L^{2}([-T,T]^{2})}^{2}  \label{In1}
\end{equation}%
We have, 
\begin{align*}
& (\tilde{h}_{T}\otimes _{1}\tilde{h}_{T})(x,y) \\
& =\int_{[-T,T]}\tilde{h}_{T}(x,z)\tilde{h}_{T}(y,z)dz \\
& =\frac{1}{4}\int_{[-T,0]}h_{T}(x,z)h_{T}(y,z)\mathbf{1}_{[0,T]}(x)\mathbf{1%
}_{[0,T]}(y)dz+\frac{1}{4}\int_{[0,T]}h_{T}(z,x)h_{T}(z,y)\mathbf{1}%
_{[-T,0]}(x)\mathbf{1}_{[-T,0]}(y)dz
\end{align*}%
Hence 
\begin{align}
& \Vert \tilde{h}_{T}\otimes _{1}\tilde{h}_{T}\Vert _{L^{2}([-T,T]^{2})}^{2}
\label{inegalite-contr} \\
& \leqslant \frac{1}{8}\int_{[0,T]^{2}}\left(
\int_{[-T,0]}h_{T}(x,z)h_{T}(y,z)\right) ^{2}dxdy+\frac{1}{8}%
\int_{[-T,0]^{2}}\left( \int_{[0,T]}h_{T}(z,x)h_{T}(z,y)\right) ^{2}dxdy 
\notag
\end{align}%
On the other hand by Fubini's theorem 
\begin{align*}
& \int_{[0,T]^{2}}\left( \int_{[-T,0]}h_{T}(x,z)h_{T}(y,z)dz\right) ^{2}dxdy=%
\frac{1}{T^{2}}\int_{[0,T]^{2}}\left[ \int_{[-T,0]}\int_{[0,T]^{2}}\bar{f}%
_{r}(x)\bar{\bar{f}}_{r}(z)\bar{f}_{s}(y)\bar{\bar{f}}_{s}(z)drdsdz\right]
^{2}dxdy \\
& =\frac{1}{T^{2}}\int_{[0,T]^{2}}\left( \int_{[0,T]^{2}}\bar{f}_{r}(x)\bar{f%
}_{s}(y)\langle \bar{\bar{f}}_{r},\bar{\bar{f}}_{s}\rangle
_{L^{2}([-T,0])}drds\right) ^{2}dxdy \\
& =\frac{1}{T^{2}}\int_{[0,T]^{2}}\int_{[0,T]^{4}}\bar{f}_{r}(x)\bar{f}%
_{s}(y)\bar{f}_{v}(x)\bar{f}_{u}(y)\langle \bar{\bar{f}}_{r},\bar{\bar{f}}%
_{s}\rangle _{L^{2}([-T,0])}\langle \bar{\bar{f}}_{v},\bar{\bar{f}}%
_{u}\rangle _{L^{2}([-T,0])}drdsdudvdxdy
\end{align*}%
Using the fact that the other term of (\ref{inegalite-contr}) can be treated
similarly and that $\langle \bar{\bar{f}}_{r},\bar{\bar{f}}_{s}\rangle
_{L^{2}([-T,0])}=\langle {\bar{f}}_{r},{\bar{f}}_{s}\rangle _{L^{2}([0,T])}=%
\mathbf{E}[X^{i}(r)X^{i}(s)]$, $i=1,2$, we get 
\begin{align*}
& \Vert \tilde{h}_{T}\otimes _{1}\tilde{h}_{T}\Vert _{L^{2}([-T,T]^{2})}^{2}
\\
& \leqslant \frac{1}{4}\frac{1}{T^{2}}\int_{[0,T]^{4}}\mathbf{E}%
[X^{i}(r)X^{i}(v)]\mathbf{E}[X^{i}(s)X^{i}(u)]\mathbf{E}[X^{i}(r)X^{i}(s)]%
\mathbf{E}[X^{i}(v)X^{i}(u)]drdsdudv
\end{align*}%
On the other hand, since for $i=1,2$, $\mathbf{E}[X^{i}(r)X^{i}(s)]=\frac{%
e^{-\theta (r+s)}}{2\theta }[e^{2\theta (r\wedge s)}-1]\leqslant \frac{1}{%
2\theta }e^{-\theta |r-s|}=\mathbf{E}[Z_{i}(r)Z_{i}(s)]:=Q(r-s)$, where $%
Z_{i}(r):=\int_{-\infty }^{r}e^{-\theta (r-t)}dW^{i}(t)$, $i=1,2$ we get 
\begin{align}
& \Vert \tilde{h}_{T}\otimes _{1}\tilde{h}_{T}\Vert _{L^{2}([-T,T]^{2})}^{2}
\\
& \leqslant \frac{1}{4}\frac{1}{T^{2}}%
\int_{[0,T]^{4}}Q(u-v)Q(v-r)Q(r-s)Q(s-u)dudvdrds \\
& =\frac{1}{4}\frac{1}{T^{2}}\int_{[0,T]^{2}}dudr\int_{\mathbb{R}%
^{2}}dvds Q _{T}(u-v) Q_{T}(s-r) Q_{T}(u-s)  Q_{T}(v-r) \\
& =\frac{1}{4}\frac{1}{T^{2}}\int_{[0,T]^{2}}dudr\int_{\mathbb{R}%
^{2}}dvds Q_{T}(y) Q_{T}(u-r-x) Q_{T}(x)  Q_{T}(u-r-y),
\end{align}%
where $Q_{T}(x):=|Q(x)|\mathbf{1}_{\{|x|\leqslant T\}}$ and we used the
change of variables $y=u-v$, $x=u-s$. Therefore, applying Young's
inequality, we can conclude 
\begin{eqnarray}
\Vert \tilde{h}_{T}\otimes _{1}\tilde{h}_{T}\Vert _{L^{2}([-T,T]^{2})}^{2}
&\leqslant &\frac{1}{4}\frac{1}{T^{2}}\int_{[0,T]^{2}}dudr(Q_{T}\ast
Q_{T})(u-r)^{2}  \notag  \label{estim-norm-contr} \\
&\leqslant &\frac{1}{4}\frac{1}{T}\int_{\mathbb{R}}(Q_{T}\ast Q_{T})(z)^{2}dz
\notag \\
&=&\frac{1}{4}\frac{1}{T}\Vert Q_{T}\ast Q_{T}\Vert _{L^{2}(\mathbb{R})}^{2}
\notag \\
&\leqslant &\frac{1}{4}\frac{1}{T}\Vert Q_{T}\Vert _{L^{4/3}(\mathbb{R}%
)}^{4}=\frac{1}{4}\frac{1}{T}\left( \int_{[-T,T]}|Q(t)|^{4/3}dt\right) ^{3/4}
\notag \\
&=&\frac{1}{4}\frac{1}{T}\left( \frac{3^{3}}{2^{7}\theta ^{7}}(1-e^{-4\theta
T/3})^{3}\right) .
\end{eqnarray}%
The desired result follows using (\ref{est-kol}) and the estimates (\ref{In1}%
), (\ref{estim-norm-contr}) and Lemma \ref{FT}.
\end{proof}

We will need the following Proposition due to Michel and Pfanzagl (1971) 
\cite{MP} in the sequel which gives upper bounds for Kolmogorov's distance
between respectively the sum and the ratio of two random variables and a
standard Gaussian random variable.

\begin{proposition}
\label{estim-kol}  Let $X$, $Y$ and $Z$ be three random variables defined on
a probability space $(\Omega, \mathcal{F}, \mathbf{P})$ such that $\mathbf{P}%
(Z>0) = 1$. Then, for all $\varepsilon > 0$, we have 

\begin{enumerate}
\item $d_{Kol}(X+Y, N) \leqslant d_{Kol}(X,N) + \mathbf{P}( |Y| >
\varepsilon) + \varepsilon$. 

\item $d_{Kol}(\frac{X}{Z}, N) \leqslant d_{Kol}(X,N) +\mathbf{P}( |Z-1| >
\varepsilon) + \varepsilon .$ 
\end{enumerate}

where $N \sim \mathcal{N}(0,1).$
\end{proposition}

\begin{proposition}
\label{exp} Let $Y$ be a r.v. such that $Y = N \times N^{\prime}$ where $N
\sim \mathcal{N}(0,\sigma^{2}_1)$ and $N^{\prime} \sim \mathcal{N}%
(0,\sigma^{2}_2)$ two independent Gaussian r.v defined on a probability
space $(\Omega, \mathcal{F}, \mathbf{P})$. Then, there exists a constant $%
\beta > \frac{2 \sqrt{3}}{3} \sigma_1 \sigma_2$ such that 
\begin{equation*}
\mathbf{E}\left[e^{\frac{Y}{\beta}} \right] < 2.
\end{equation*}
Moreover, there exists a constant $C> \frac{\sqrt{3}}{3} \pi $ such that $%
\beta < C \times \mathbf{E}[|Y|]$.
\end{proposition}
\begin{proof}
By the independence of $N$ and $N^{\prime}$, it's easy to check that for any 
$\beta > \sigma_1 \sigma_2$, we have 
\begin{align*}
\mathbf{E}\left[e^{\frac{Y}{\beta}} \right] & = \frac{1}{2 \pi \sigma_1
\sigma_2} \int_{\mathbb{R}^2} e^{\frac{xy}{\beta}} e^{-\frac{x^2}{2
\sigma^2_1}} e^{-\frac{y^2}{2 \sigma^2_2}} dx dy \\
& = \frac{1}{\sqrt{(1- \frac{\sigma^2_1 \sigma^2_2}{\beta^2})}}.
\end{align*}
Thus the constraint $\mathbf{E}\left[e^{\frac{Y}{\beta}} \right] < 2$
implies that $\beta$ should be such that $\beta > \frac{2 \sqrt{3}}{3}
\sigma_1 \sigma_2$. On the other hand since $\mathbf{E}[|Y|] = \mathbf{E}%
[|N|] \times \mathbf{E}[|N^{\prime}|] = \frac{2 \sigma_1 \sigma_2}{\pi}$,
thus there exists a constant $C > \frac{\sqrt{3}}{3} \pi$ such that $\beta <
C \times \mathbf{E}[|Y|]$.
\end{proof}

\begin{corollary}
\label{kol-sum} Let $X$,$Y$ be two r.v. defined on a probability space $%
(\Omega ,\mathcal{F},\mathbf{P})$ such that $Y=N\times N^{\prime }$ where $%
N\sim \mathcal{N}(0,\sigma _{1}^{2})$ and $N^{\prime }\sim \mathcal{N}%
(0,\sigma _{2}^{2})$. Then, there exists a constant $\frac{2\sqrt{3}}{3}%
\sigma _{1}\sigma _{2}<\beta \leqslant \left( (C\times \mathbf{E}%
[|Y|])\wedge 4\right) $ such that 
\begin{equation*}
d_{Kol}(X+Y,\mathcal{N}(0,1))\leqslant d_{Kol}(X,\mathcal{N}(0,1))+\beta
\left( 1+\ln \left( \frac{4}{\beta }\right) \right) .
\end{equation*}
\end{corollary}
\begin{proof}
Let $X$, $Y$ be two random variables , then from Michel and Pfanzagl (1971),
for all $\varepsilon > 0$, 
\begin{equation*}
d_{Kol}(X+Y, \mathcal{N}(0,1)) \leqslant d_{Kol}(X,\mathcal{N}(0,1)) + 
\mathbf{P}( |Y| > \varepsilon) + \varepsilon.
\end{equation*}
Since $Y = N \times N^{\prime}$ where $N \sim \mathcal{N}(0,\sigma^{2}_1)$
and $N^{\prime} \sim \mathcal{N}(0,\sigma^{2}_2)$ then, by Proposition \ref%
{exp} and Markov's inequality, we have 
\begin{equation*}
\mathbf{P}( |Y| > \varepsilon) = 2 \times \mathbf{P}(Y > \varepsilon)
\leqslant 2 \times \mathbf{E}\left[e^{\frac{Y}{\beta}}\right] e^{-\frac{%
\varepsilon}{\beta}} < 4 e^{-\frac{\varepsilon}{\beta}}.
\end{equation*}
Thus we can write 
\begin{equation*}
d_{Kol}(X+Y, \mathcal{N}(0,1)) \leqslant d_{Kol}(X, \mathcal{N}(0,1)) +
\inf_{\varepsilon > 0} g_{\beta}(\varepsilon).
\end{equation*}
where $g_{\beta}(\varepsilon) := 4 e^{-\frac{\varepsilon}{\beta}}+
\varepsilon.$ Since $g_{\beta}$ is convex on $\mathbb{R}_{+}$, $\arg
\inf\limits_{\varepsilon > 0} g_{\beta}(\varepsilon) =
\varepsilon^{*}(\beta) = \beta \ln(\frac{4}{\beta})$, $\beta < \left( (C
\times \mathbf{E}[|Y|]) \wedge 4 \right)$, with $C$ the constant from
Proposition \ref{exp}. The desired result follows.
\end{proof}

To prove the convergence in law of $\frac{Y_{12}(T)}{\sqrt{T}}$, recall that 
$T^{-1/2}Y_{12}(T)=F_{T}-\sqrt{T}\bar{X_{1}}(T)\bar{X_{2}}(T).$ We can write 
$\bar{X_{i}}(T):=I_{1}^{W_{i}}(g_{T})$, $i=1,2$ where $g_{T}:=T^{-1}%
\int_{0}^{T}f_{t}dt$ and we have for $i=1,2$ 
\begin{align}
\mathbf{E}[\bar{X_{i}}^{2}(T)]& =\Vert g_{T}\Vert _{L^{2}([0,T])}^{2}  \notag
\\
& =\frac{1}{T^{2}}\int_{0}^{T}(\int_{0}^{T}f_{t}(u)dt)^{2}du  \notag \\
& =\frac{1}{T^{2}}\int_{0}^{T}e^{2\theta u}(\int_{u}^{T}e^{-\theta
t}dt)^{2}du  \notag \\
& =\frac{1}{T^{2}}\frac{1}{\theta ^{2}}\int_{0}^{T}(1-e^{-\theta
(T-u)})^{2}du\leqslant \frac{1}{\theta ^{2}}\frac{1}{T}.  \label{normbarX}
\end{align}%
Hence by the independence of $X_{1}$ and $X_{2}$ and denoting $Y(T):=\sqrt{T}%
\bar{X_{1}}(T)\bar{X_{2}}(T)$, we get 
\begin{equation}
\mathbf{E}[Y(T)^{2}]\leqslant \frac{1}{T}\frac{1}{\theta ^{4}}.
\label{estil-term2-0}
\end{equation}%
Then in virtue of Proposition \ref{exp} and Corollary \ref{kol-sum}, there
exists a constant $\beta $ with $\frac{4\sqrt{3}}{3}\sqrt{T}\theta
^{3/2}\Vert g_{T}\Vert _{L^{2}([0,T])}^{2}<\beta <\frac{C}{\sqrt{T}\theta
^{2}}\wedge 4$ such that 
\begin{equation*}
d_{Kol}\left( \frac{Y_{12}(T)}{\sqrt{T}},\mathcal{N}(0,\frac{1}{4\theta ^{3}}%
)\right) \leqslant d_{Kol}\left( F_{T},\mathcal{N}(0,\frac{1}{4\theta ^{3}}%
)\right) +\beta \left( 1+\ln \left( \frac{4}{\beta }\right) \right) .
\end{equation*}%
On the other hand, since the function $x\mapsto x(1+\ln (\frac{4}{x}))$ is
increasing on $(0,4)$, we have for $T$ large enough

\begin{equation*}
d_{Kol}\left( \frac{Y_{12}(T)}{\sqrt{T}}, \mathcal{N}(0,\frac{1}{4 \theta^3}%
) \right) \leqslant d_{Kol}\left( F_T, \mathcal{N}(0,\frac{1}{4 \theta^3})
\right) + \frac{c(\theta)}{2} \frac{\ln(T)}{\sqrt{T}}.
\end{equation*}
The following proposition follows.

\begin{proposition}
\label{estim-num} There exists a constant $C(\theta)$ depending only on $%
\theta$, such that 
\begin{equation*}
d_{Kol}\left(\frac{Y_{12}(T)}{\sqrt{T}}, \mathcal{N}\left(0,\frac{1}{4
\theta^{3}}\right) \right) \leqslant C(\theta) \times \frac{\ln(T)}{\sqrt{T}}%
.
\end{equation*}
In particular, $\frac{Y_{12}(T)}{\sqrt{T}} \overset{\mathrm{\mathcal{L}}}{%
\longrightarrow }\mathcal{N}(0,\frac{1}{4 \theta^{3}} )$ as $T \rightarrow
+\infty$.
\end{proposition}
Having just completed the study of the convergence in law of the numerator
in $\rho (T)$, in order to study the convergence in law of $\sqrt{T}\rho (T)$%
, we will use Proposition \ref{estim-kol} assertion 2 and the fact that 
\begin{equation}
\sqrt{\theta }\sqrt{T}\rho (T)=\frac{2\theta ^{3/2}\frac{Y_{12}(T)}{\sqrt{T}}%
}{2\theta \sqrt{\frac{Y_{11}(T)}{T}\times \frac{Y_{22}(T)}{T}}}
\label{decom-rho}
\end{equation}%
to show in the next subsection that the denominator concentrates to the
value 1, and that the behavior of $\sqrt{T}\rho (T)$ is thus given by that
of the numerator above.

\subsection{The denominator term}

Let us denote the denominator term 
\begin{equation}
D(T):=D:=2\theta \sqrt{\frac{Y_{11}(T)\times Y_{22}(T)}{T\times T}},
\label{Zdenom}
\end{equation}%
According to Proposition \ref{estim-kol} assertion 2. we need to estimate $%
\mathbf{P}(|D-1|>\varepsilon )$ for instance for $0<\varepsilon <1$. Using
the fact that $D\geq 0$ a.s. then $|D-1|\leqslant |D^{2}-1|$ a.s. Now using
the shorthand notation $\bar{Y}_{ii}(T):=\frac{Y_{ii}(T)}{T/2\theta }$, $%
i=1,2$, thus, we have a.s. 
\begin{align}
|D-1|\leqslant |D^{2}-1|& \leqslant |\bar{Y}_{11}(T)\bar{Y}_{22}(T)-1| 
\notag \\
& \leqslant |\bar{Y}_{11}(T)-1|\times |\bar{Y}_{22}(T)-1|+|\bar{Y}%
_{11}(T)-1|+|\bar{Y}_{22}(T)-1|  \label{estim-Z-continu}
\end{align}%
Thus, using the fact that $\bar{Y}_{11}(T)$ and $\bar{Y}_{22}(T)$ are equal
in law, we get for any $\varepsilon <1$, 
\begin{align}
\mathbf{P}\left(|D-1|>\varepsilon \right) & \leqslant \mathbf{P}\left( |%
\bar{Y}_{11}(T)-1|>\frac{\varepsilon }{3}\right) +\mathbf{P}\left( |\bar{Y}%
_{22}(T)-1|>\frac{\varepsilon }{3}\right) +2\mathbf{P}\left( |\bar{Y}%
_{11}(T)-1|^{2}>\frac{\varepsilon }{3}\right)   \notag \\
& \leqslant 4\times \mathbf{P}\left( |\bar{Y}_{11}(T)-1|>\frac{\varepsilon }{%
3}\right) .  \label{estimeD}
\end{align}%
We lighten the notation by writing $\varepsilon $ instead of $\varepsilon /3$%
, therefore by Proposition \ref{estim-kol} assertion 2. applied to $\rho (T)$
in (\ref{decom-rho}), we get 
\begin{equation}
d_{Kol}\left( \sqrt{\theta }\sqrt{T}\rho (T),N\right) \leqslant
d_{Kol}\left( 2\theta ^{3/2}\frac{Y_{12}(T)}{\sqrt{T}},N\right) +4\times 
\mathbf{P}\left( |\bar{Y}_{11}(T)-1|>\varepsilon \right) +3\varepsilon .
\label{rhoestim-1}
\end{equation}%
where $N\sim \mathcal{N}(0,1)$.

The next step is to control the term $\mathbf{P}\left( |\bar{Y}%
_{11}(T)-1|>\varepsilon \right) $. We have : 
\begin{align*}
\bar{Y}_{11}(T)& =\frac{2\theta }{T}\int_{0}^{T}\left( X_{1}^{2}(u)-\mathbf{E%
}[X_{1}^{2}(u)]\right) du+\frac{2\theta }{T}\int_{0}^{T}\mathbf{E}%
[X_{1}^{2}(u)]du-2\theta \bar{X}_{1}^{2}(T) \\
& =\frac{2\theta }{T}\int_{0}^{T}\left( (I_{1}^{W_{1}}(f_{u}))^{2}-\Vert
f_{u}\Vert _{L^{2}([0,T])}^{2}\right) du+\frac{1}{T}\int_{0}^{T}\mathbf{E}%
[X_{1}^{2}(u)]du-\bar{X}_{1}^{2}(T) \\
& =2\theta I_{2}^{W_{1}}(k_{T})+\frac{2\theta }{T}\int_{0}^{T}\mathbf{E}%
[X_{1}^{2}(u)]du-2\theta \bar{X}_{1}^{2}(T) \\
& :=A_{\theta }(T)+\mu _{\theta }(T)-2\theta \bar{X}_{1}^{2}(T).
\end{align*}%
where 
\begin{align*}
k_{T}(x,y)& :=\frac{1}{T}\int_{0}^{T}f_{u}^{\otimes 2}(x,y)du \\
& =\frac{1}{T}\int_{0}^{T}e^{-\theta (u-x)}e^{-\theta (u-y)}\mathbf{1}%
_{[0,u]}(x)\mathbf{1}_{[0,u]}(y)du \\
& =\frac{1}{T}\frac{1}{2\theta }e^{\theta x}e^{\theta y}\left( e^{-2\theta
(x\vee y)}-e^{-2\theta T}\right) \mathbf{1}_{[0,T]}(x)\mathbf{1}_{[0,T]}(y)
\end{align*}%
and 
\begin{align*}
\mu _{\theta }(T)& =\frac{2\theta }{T}\int_{0}^{T}\mathbf{E}[X_{1}^{2}(u)]du=%
\frac{2\theta }{T}\int_{0}^{T}\Vert f_{u}\Vert _{L^{2}([0,T]}^{2}du \\
& =\frac{2\theta }{T}\int_{0}^{T}\int_{0}^{T}f_{u}^{2}(t)dtdu \\
& =\frac{2\theta }{T}\int_{0}^{T}\int_{0}^{u}e^{-2\theta (u-t)}dtdu \\
& =\frac{1}{T}\int_{0}^{T}(1-e^{-2\theta u})du \\
& =1-\frac{1}{2\theta T}\left( 1-e^{-2\theta T}\right) 
\end{align*}%
Then we immediately get the mean concentration around 1:%
\begin{equation}
|\mu _{\theta }(T)-1|\leqslant \frac{1}{2\theta T}=O(\frac{1}{T}).
\label{muthetaub}
\end{equation}
For the term $A_{\theta }(T)$, in a similar way to the calculus in the proof
of Proposition \ref{varianceFT} and since $k_{T}$ is symmetric, we get 
\begin{align*}
\mathbf{E}\left[ I_{2}^{W_{1}}(k_{T})^{2}\right] & =2\Vert k_{T}\Vert
_{L^{2}([0,T]^{2})}^{2} \\
& =\frac{1}{T^{2}}\frac{1}{2\theta ^{2}}\int_{[0,T]^{2}}e^{2\theta
x}e^{2\theta y}\left( e^{2\theta (x\vee y)}-e^{-2\theta T}\right) ^{2}dxdy \\
& =\frac{1}{T^{2}}\frac{1}{2\theta ^{3}}\left( \frac{1}{4\theta }%
(1-e^{-4\theta T})+\frac{1}{\theta }(e^{-2\theta T}-1)+T(1+2e^{-2\theta T})-%
\frac{1}{2\theta }(1-e^{-4\theta T})\right)  \\
& \leqslant \frac{1}{2\theta ^{3}}\left( 3+\frac{7}{4\theta }\right) \frac{1%
}{T}.
\end{align*}

Therefore, we have $\text{Var}(A_{\theta }(T))=O(\frac{1}{T})$, since 
\begin{align*}
\text{Var}(A_{\theta }(T))& =\text{Var}(2\theta I_{2}^{W_{1}}(k_{T})^{2}) \\
& =4\theta ^{2}\mathbf{E}\left[ I_{2}^{W_{1}}(k_{T})^{2}\right] \leqslant 
\frac{2}{\theta }\left( 3+\frac{7}{4\theta }\right) \frac{1}{T}.
\end{align*}%
Finally, by equation (\ref{normbarX}), $2\theta \mathbf{E}[\bar{X}%
_{1}^{2}(T)]=O(\frac{1}{T})$. On the other hand, we can write 
\begin{equation*}
\bar{Y}_{11}(T)=\tilde{Y}_{11}(T)+y_{11}(T)
\end{equation*}%
where 
\begin{equation*}
\left\{ 
\begin{array}{ll}
\tilde{Y}_{11}(T):=A_{\theta }(T)-2\theta \left( \bar{X}_{1}^{2}(T)-\mathbf{E%
}\left[ \bar{X}_{1}^{2}(T)\right] \right) , &  \\ 
~~ &  \\ 
y_{11}(T):=\mu _{\theta }(T)-2\theta \mathbf{E}\left[ \bar{X}_{1}^{2}(T)%
\right] . & 
\end{array}%
\right. 
\end{equation*}%
By the product formula (\ref{productformula01}), the r.v. $\tilde{Y}_{11}(T)$
belongs to the second Wiener chaos while $y_{11}(T)$ is deterministic.
Moreover, we have 
\begin{align}
\text{Var}(\tilde{Y}_{11}(T))& =\text{Var}\left( A_{\theta }(T)-2\theta \bar{%
X}_{1}^{2}(T)\right)   \notag \\
& \leqslant 2\text{Var}\left( A_{\theta }(T)\right) +8\theta ^{2}\text{Var}%
\left( \bar{X}_{1}^{2}(T)\right)   \notag \\
& \leqslant \frac{4}{T\theta }\left( 3+\frac{7}{4\theta }\right) +\frac{8^{2}%
}{\theta ^{2}T^{2}}\leqslant \frac{cst(\theta )}{T}:=\left[ \frac{4}{\theta }%
\left( 3+\frac{7}{4\theta }\right) +\frac{8^{2}}{\theta ^{2}}\right] \frac{1%
}{T},  \label{Y11}
\end{align}%
where we used the hypercontractivity property (\ref{hypercontractivity}) on
Wiener chaos for $\bar{X}_{1}^{2}(T)$, since $\text{Var}(\bar{X}_{1}^{2}(T))=%
\mathbf{E}[I_{1}^{W_{1}}(g_{T})^{4}]-\left( \mathbf{E}%
[I_{1}^{W_{1}}(g_{T})^{2}]\right) ^{2}$ and $\mathbf{E}%
[I_{1}^{W_{1}}(g_{T})^{2}]=\mathbf{E}[\bar{X}_{1}^{2}(T)]\leqslant \frac{1}{%
T\theta ^{2}}$. The last estimate plus the estimate (\ref{muthetaub}) on $%
\mu _{\theta }(T)$ established earlier imply that 
\begin{equation}
\left\vert 1-y_{11}(T)\right\vert \leq \frac{5}{2}\frac{1}{\theta T}.
\label{y11}
\end{equation}

With all those estimates in place, we return to our main target to control $%
\mathbf{P}\left( |\bar{Y}_{11}(T)-1|>\varepsilon \right) $ for some $%
0<\varepsilon <1$, Or 
\begin{align*}
\mathbf{P}\left( |\bar{Y}_{11}(T)-1|>\varepsilon \right) & =\mathbf{P}\left(
|\tilde{Y}_{11}(T)+y_{11}(T)-1|>\varepsilon \right)  \\
& =\mathbf{P}\left( \left\{ \tilde{Y}_{11}(T)>\varepsilon
+1-y_{11}(T)\right\} \cup \left\{ -\tilde{Y}_{11}(T)>\varepsilon
+y_{11}(T)-1\right\} \right) 
\end{align*}%
Of course second chaos r.v. are not symmetric, but since we do not know the
sign of $1-y_{11}(T)$ and second Wiener chaos can be skewed in either
direction, there is no loss of efficiency to treat $\tilde{Y}_{11}(T)$ and $-%
\tilde{Y}_{11}(T)$ in the same fashion. Recall the from Proposition 2.7.13
of \cite{NP-book} that any second chaos r.v. $F$ has the following
representation 
\begin{equation*}
F=\sum\limits_{n=1}^{+\infty }\lambda _{n}\left( Z_{n}^{2}-1\right) 
\end{equation*}%
where $\left\{ \lambda _{n},n\geq 1\right\} $ is a sequence of reals for
which $\left\vert \lambda _{n}\right\vert $ is decreasing and $%
(Z_{n})_{n\geq 1}$ are independent standard Gaussian random variables and 
\begin{equation*}
\text{Var}(F)=\sum\limits_{n=1}^{+\infty }|\lambda _{n}|^{2}<+\infty .
\end{equation*}%
Moreover, by the product formula (\ref{productformula01}), if $F$ is a
quadratic functional of a Gaussian process, then $\sum\limits_{n=1}^{+\infty
}\lambda _{n}$ is the expectation of that functional. Therefore with $F=%
\tilde{Y}_{11}(T)$, there exists $\left\{ \lambda _{n}(T),n\geq 1\right\} $
and $(Z_{n})_{n\geq 1}$ iid $\mathcal{N}(0,1)$, such that 
\begin{equation*}
\tilde{Y}_{11}(T)=\sum\limits_{n=1}^{+\infty }\lambda _{n}(T)\left(
Z_{n}^{2}-1\right) .
\end{equation*}%
One immediately checks that the expression to be added to $\tilde{Y}_{11}(T)$
to make it a quadratic functional is $\mu _{\theta }(T)-2\theta \mathbf{E}[%
\bar{X}_{1}^{2}(T)]$ which is equal to $\sum\limits_{n=1}^{+\infty }\lambda
_{n}(T)$. Therefore, the sequence $\left\{ \lambda _{n}(T),n\geq 1\right\} $
satisfies 
\begin{equation*}
\left\{ 
\begin{array}{ll}
\sum\limits_{n=1}^{+\infty }\lambda _{n}^{2}(T)=\text{Var}(\tilde{Y}%
_{11}(T))\leqslant \frac{c(\theta )}{T}, &  \\ 
&  \\ 
|\sum\limits_{n=1}^{+\infty }\lambda _{n}(T)|=1+O(\frac{1}{T}). & 
\end{array}%
\right. 
\end{equation*}%
From the representation we just established, we will set a general global
tail for any r.v. in the second Wiener chaos, which is convenient for our
purposes. Let $Y=\sum\limits_{n=1}^{+\infty }\lambda _{n}(Z_{n}^{2}-1)$ let $%
\sigma :=\sqrt{\text{Var}(Y)}=\sqrt{\sum\limits_{n=1}^{+\infty }|\lambda
_{n}|^{2}}$ and $v=\sum\limits_{n=1}^{+\infty }\lambda _{n}$. Assume that $%
|v|<+\infty $. Let $\beta >0$ which is a constant to be chosen later. Then
by Markov's inequality, we have for all $y$ 
\begin{align}
\mathbf{P}\left( Y>y\right) & \leqslant e^{-\frac{y}{\beta }}\times \mathbf{E%
}\left[ e^{\frac{Y}{\beta }}\right] =e^{-\frac{y}{\beta }}\times \mathbf{E}%
\left[ e^{\frac{1}{\beta }\sum\limits_{n=1}^{+\infty }\lambda
_{n}(Z_{n}^{2}-1)}\right]   \notag \\
& =e^{-\frac{y}{\beta }}\prod_{n=1}^{+\infty }\mathbf{E}\left[ e^{\lambda
_{n}\frac{Z_{n}^{2}}{\beta }}\right] \times e^{-\frac{v}{\beta }}  \notag \\
& =e^{-\frac{y+v}{\beta }}\times \left( \prod_{n=1}^{+\infty }\frac{1}{\sqrt{%
1-\frac{2\lambda _{n}}{\beta }}}\right) ,  \label{produit2}
\end{align}%
This formula requires that $\frac{2\lambda _{n}}{\beta }<1$ for all $n\geq 1$%
, but since the sign of $\lambda _{n}$ is unknown and $|\lambda _{n}|$
decreases, it is sufficient to require that $\beta >2|\lambda _{1}|$. Since $%
\sum\limits_{n=1}^{+\infty }|\lambda _{n}|^{2}=\text{Var}(Y)$, we can say
that $|\lambda _{1}|<\sqrt{\text{Var}(Y)}=\sigma $. Therefore, to be
completely safe we require that $\beta \geq 4\sigma $.We must also check
that the product in (\ref{produit2}) converges. In fact, notice that since $%
\forall u\in \lbrack 0,\frac{1}{2}],-\ln (1-u)\leqslant u+\frac{u^{2}}{2}%
+u^{3}$, thus as $0\leqslant \frac{2\lambda _{n}}{\beta }\leqslant \frac{1}{2%
}$, for all $n$, we have 
\begin{equation*}
\ln \left( \prod_{n=1}^{+\infty }\frac{1}{\sqrt{1-\frac{2\lambda _{n}}{\beta 
}}}\right) \leqslant \frac{1}{2}\sum\limits_{n=1}^{+\infty }\left( \frac{%
2\lambda _{n}}{\beta }+\frac{2\lambda _{n}^{2}}{\beta ^{2}}+\frac{8\lambda
_{n}^{3}}{\beta ^{3}}\right) 
\end{equation*}%
Thus, we get 
\begin{equation}
\mathbf{P}\left( Y>y\right) \leqslant e^{-\frac{y}{\beta }}\times e^{\frac{%
\text{Var}(Y)}{\beta ^{2}}}\times e^{\frac{k_{3}(Y)}{2\beta ^{3}}}.
\label{estimate-y}
\end{equation}%
where we used the fact that ( see Proposition 2.7.13 of \cite{NP-book}) 
\begin{equation*}
\sum\limits_{n=1}^{+\infty }\lambda _{n}^{3}=\frac{1}{8}k_{3}(Y),
\end{equation*}%
where $k_{3}(Y)$ denotes the third cumulant of $Y$. Since $Y$ is centered
then this cumulant is equal to the third moment: $k_{3}(Y)=\mathbf{E}[Y^{3}]$%
.

Applying inequality (\ref{estimate-y}) to $Y=\tilde{Y}_{11}(T)$ where 
\begin{equation*}
\sigma \leqslant \sqrt{\frac{cst(\theta )}{T}}\quad \text{and}\quad |v|=1+O(%
\frac{1}{T}),
\end{equation*}%
we can pick any $\beta \geq 4\sigma $, thus we can chose $\beta =4\sqrt{%
\frac{cst(\theta )}{T}}$. This implies that $\frac{\text{Var}(\tilde{Y}%
_{11}(T))}{\beta ^{2}}=\frac{\sigma ^{2}}{\beta ^{2}}\leqslant \frac{1}{16}$%
. On the other hand, since $\tilde{Y}_{11}(T)$ is in the second Wiener
chaos, then by the hypercontractivity property in Section 2, $k_{3}(\tilde{Y}%
_{11}(T))\leqslant 9\text{Var}(\tilde{Y}_{11}(T))^{3/2}$, and consequently $%
\frac{1}{2}\frac{k_{3}(\tilde{Y}_{11}(T))}{\beta ^{3}}\leqslant \frac{9}{2}%
\times \left( \sum\limits_{n=1}^{+\infty }\lambda _{n}^{2}/\beta ^{2}\right)
^{3/2}$. Finally, since $\beta ^{2}>16\sum\limits_{n=1}^{+\infty }\lambda
_{n}^{2}$, thus we get the following: 
\begin{equation}
\mathbf{P}(\tilde{Y}_{11}(T)>y)\leqslant K\times \exp \left( -\frac{y\sqrt{T}%
}{4\sqrt{cst(\theta )}}\right)   \label{estimateYt}
\end{equation}%
where 
\begin{equation*}
K:=\exp \frac{1}{16}\left( 1+\frac{9}{2}\times \frac{1}{\sqrt{16}}\right)
=e^{17/128}.
\end{equation*}
We now replace $y$ by $\varepsilon +1-y_{11}(T)$. Note that we must also
evaluate $\mathbf{P}(-\tilde{Y}_{11}(T)>\varepsilon +1-y_{11}(T))$ but since
the signs of $\lambda _{n}$ and $1-y_{11}(T)$ are not known, this will yield
exactly to the same estimate as $\mathbf{P}(\tilde{Y}_{11}(T)>\varepsilon
+1-y_{11}(T))$. Thus from the estimate (\ref{estimateYt}), we get 
\begin{equation*}
\mathbf{P}\left( |\bar{Y}_{11}(T)-1|>\varepsilon \right) \leqslant 2K\exp
\left( -\frac{y\sqrt{T}}{4\sqrt{cst(\theta )}}\right) 
\end{equation*}%
Let us denote $c:=4\sqrt{cst(\theta )}$, we must choose $\varepsilon $. Let 
\begin{equation}
\varepsilon =d\times \frac{\ln (T)}{\sqrt{T}},  \label{epsilonchoice}
\end{equation}%
where $d$ is some constant to be chosen as well. Thus since we proved in (%
\ref{y11}) that $|1-y_{11}|\leqslant \frac{5}{2}\frac{1}{\theta T}$, we get 
\begin{equation*}
y:=\varepsilon +1-y_{11}(T)\geq \varepsilon -|1-y_{11}(T)|\geq d\times \frac{%
\ln (T)}{\sqrt{T}}-\frac{5}{2\theta T}.
\end{equation*}%
For $T$ large enough, for instance for $T>T^{\ast }=e\vee \left( \frac{5}{%
d\theta }\right) ^{2}$, we have 
\begin{equation*}
d\times \frac{\ln (T)}{\sqrt{T}}>\frac{5}{\theta T}
\end{equation*}%
so that 
\begin{equation*}
y>d\times \frac{\ln (T)}{2\sqrt{T}}
\end{equation*}%
and hence 
\begin{equation*}
\exp \left( -\frac{y\sqrt{T}}{4\sqrt{cst(\theta )}}\right) \leqslant e^{-d%
\frac{\ln (T)}{2c}}\exp \left( -\frac{d\ln T}{2c}\right) =T^{-d/(2c)}.
\end{equation*}%
Thus it's sufficient to choose $d=c$, obtaining 
\begin{equation}
\mathbf{P}\left( |\bar{Y}_{11}(T)-1|>\varepsilon \right) +3\varepsilon
\leqslant \frac{K}{\sqrt{T}}+3c\frac{\ln (T)}{\sqrt{T}}\sim 3c\frac{\ln (T)}{%
\sqrt{T}}\quad \text{for }T\quad \text{large}  \label{majorY11}
\end{equation}%
Summarizing, with $K=e^{17/128}$ and $c=4\sqrt{cst(\theta )}=4\sqrt{%
4(3+7/(4\theta ))/\theta +64/\theta ^{2}}$, for $T>T^{\ast }\left( \theta
\right) :=e\vee (\frac{5}{4\theta \sqrt{cst\left( \theta \right) }})^{2}$,
then according to inequality (\ref{estim-Z-continu}), with the choice for $%
\varepsilon $ in (\ref{epsilonchoice}) with $d=c$, we have the following
estimate for the denominator term $D(T)$ defined in (\ref{Zdenom}), as
follows 
\begin{equation*}
\mathbf{P}(|D-1|>\varepsilon )+\varepsilon \leqslant \frac{4K}{\sqrt{T}}+%
\frac{12c\ln (T)}{\sqrt{T}},\quad \text{for }T>T^{\ast }\left( \theta
\right) :=\max \left( e,\frac{25}{16\theta ^{2}cst\left( \theta \right) }%
\right) .
\end{equation*}%
Thus, from inequalities (\ref{rhoestim-1}) and (\ref{majorY11}), we get the
following theorem for the convergence in law of Yule's statistic $\rho (T)$
as $T\rightarrow +\infty $.

\begin{theorem}
\label{rho-case-continuous} There exists a constant $c(\theta )$ depending
only on $\theta $ such that for $T$ large enough, we have 
\begin{equation*}
d_{Kol}\left( \sqrt{\theta }\sqrt{T}\rho (T),\mathcal{N}(0,1)\right)
\leqslant c(\theta )\frac{\ln (T)}{\sqrt{T}}.
\end{equation*}%
In particular, 
\begin{equation*}
\sqrt{\theta }\sqrt{T}\rho (T)\overset{\mathrm{\mathcal{L}}}{\longrightarrow 
}\mathcal{N}(0,1),\quad \text{as}\quad T\rightarrow +\infty .
\end{equation*}
\end{theorem}

\begin{remark}
From expression $cst(\theta )=4(3+7/(4\theta ))/\theta +64/\theta ^{2}$ at
the end of the calculation preceding Theorem \ref{rho-case-continuous}, and
similar estimates elsewhere above, a detailed analysis of how these
constants depend on $\theta $ show that for $\theta >1$,  
\begin{equation*}
0<c(\theta )<\frac{c_{u}}{\sqrt{\theta }},
\end{equation*}%
where $c_{u}$ is a universal constant. This analysis is omitted for
conciseness. In particular, $c(\theta )\rightarrow 0$ as $\theta \rightarrow
+\infty $.
\end{remark}

\section{Discrete observations}

We assume now that the pair $(X_{1},X_{2})$ of Ornstein-Uhlenbeck processes
is observed at $n$ equally spaced discrete time instants $t_{k}:=k\times
\Delta _{n}$, where $\Delta _{n}$ is the observation mesh and $%
T_{n}:=n\Delta _{n}$ is the length of the \textquotedblleft observation
window\textquotedblright . We assume $\Delta _{n}\rightarrow 0$ and $%
T_{n}\rightarrow +\infty $, as $n\rightarrow +\infty $. The aim of this
section is to prove a CLT for the following statistic $\tilde{\rho}(n)$
which can be considered as a discrete version of Yule's nonsense correlation
statistic $\rho (T)$, defined by 
\begin{equation}
\tilde{\rho}(n):=\frac{\tilde{Y}_{12}(n)}{\sqrt{\tilde{Y}_{11}(n)\times 
\tilde{Y}_{22}(n)}}  \label{ro-discrete}
\end{equation}%
where $\tilde{Y}_{ij}(n)$, $i,j=1,2$ are the Riemann-type discretization of $%
Y_{ij}(T)$ defined as follows 
\begin{equation}
\tilde{Y}_{ij}(n):=\Delta
_{n}\sum\limits_{k=0}^{n-1}X_{i}(t_{k})X_{j}(t_{k})-T_{n}\tilde{X}_{i}(n)%
\tilde{X}_{j}(n),\quad i,j=1,2,  \label{Yij-discrete}
\end{equation}%
with $\tilde{X}_{i}(n)$ denoting the empirical mean-process of $X_{i}$, $%
i=1,2$, namely 
\begin{equation*}
\tilde{X}_{i}(n):=\frac{1}{n}\sum\limits_{k=0}^{n-1}X_{i}(t_{k}),\quad i=1,2
\end{equation*}%
As in the continuous case, we make use of the following expression of $%
\tilde{\rho}(n)$ along with Proposition \ref{estim-kol} in order to prove
its convergence in law to a Gaussian distribution: 
\begin{equation}
\sqrt{\theta }\sqrt{T_{n}}\tilde{\rho}(n)=\frac{2\theta ^{3/2}\frac{\tilde{Y}%
_{12}(n)}{\sqrt{T_{n}}}}{2\theta \sqrt{\frac{\tilde{Y}_{11}(n)}{T_{n}}\times 
\frac{\tilde{Y}_{22}(n)}{T_{n}}}}.  \label{decomp-rho-n}
\end{equation}

\subsection{Convergence in law of $\frac{\tilde{Y}_{12}(n)}{\protect\sqrt{T_n%
}}$}

From the expression of $\tilde{Y}_{12}(n)$ given in (\ref{Yij-discrete}), we
can write 
\begin{equation}
\frac{\tilde{Y}_{12}(n)}{\sqrt{T}_{n}}=A(n)-B(n),  \label{decomY12-discret}
\end{equation}%
where 
\begin{equation*}
A(n):=\frac{\sqrt{T}_{n}}{n}\sum\limits_{k=0}^{n-1}X_{1}(t_{k})X_{2}(t_{k})%
\text{ \ }\text{and}\text{ \ }B(n):=\sqrt{T_{n}}\tilde{X}_{1}(n)\tilde{X}%
_{2}(n).
\end{equation*}%
We also defined the following random sequence 
\begin{equation*}
\delta (n):=A(n)-F_{T_{n}},
\end{equation*}%
where $F_{T_{n}}$ is defined in (\ref{FT}). The following lemma holds.

\begin{lemma}
\label{norm-delta}  Assume that $\Delta_n \rightarrow 0$ as $n \rightarrow
+\infty$ and that $T_n = n \Delta_n \rightarrow +\infty$, then, there exists
a constant $C_{\theta} := 4 \times \max\left( \frac{8}{9 \theta}, \frac{%
\sqrt{2}}{3} \frac{1}{\theta^{1/2}}, \frac{1}{4}\right)$, such that  
\begin{equation*}
\mathbf{E}[\delta^2(n)] \leqslant C_\theta \times n \Delta^2_n, 
\end{equation*}
In particular if $n \Delta^2_n \rightarrow 0$ as $n \rightarrow +\infty$, $%
\mathbf{E}[\delta^2(n)] \rightarrow 0 $, as $n \rightarrow +\infty$.
\end{lemma}
\begin{proof}
We have 
\begin{align*}
\delta (n)& =A(n)-F_{T_{n}} \\
& =\sqrt{T_{n}}\left( \frac{1}{n}\sum%
\limits_{k=0}^{n-1}X_{1}(t_{k})X_{2}(t_{k})-\frac{1}{T_{n}}%
\int_{0}^{T_{n}}X_{1}(u)X_{2}(u)du\right)  \\
& =\frac{1}{\sqrt{T_{n}}}\sum\limits_{k=0}^{n-1}%
\int_{t_{k}}^{t_{k+1}}(X_{1}(t_{k})X_{2}(t_{k})-X_{1}(u)X_{2}(u))du
\end{align*}%
By Cauchy Schwartz inequality and the fact that $\sup\limits_{t\geq 0}%
\mathbf{E}[X_{i}^{2}(t)]=\frac{1}{2\theta }$, $i=1,2$, we get 
\begin{align*}
& \mathbf{E}[\delta ^{2}(n)] \\
& =\frac{1}{{T_{n}}}\sum\limits_{k_{1},k_{2}=0}^{n-1}%
\int_{t_{k_{1}}}^{t_{k_{1}+1}}\int_{t_{k_{2}}}^{t_{k_{2}+1}}\mathbf{E}\left[
(X_{1}(t_{k_{1}})X_{2}(t_{k_{1}})-X_{1}(u)X_{2}(u))(X_{1}(t_{k_{2}})X_{2}(t_{k_{2}})-X_{1}(v)X_{2}(v))%
\right] dudv \\
& \leqslant \frac{1}{{T_{n}}}\sum\limits_{k_{1},k_{2}=0}^{n-1}%
\int_{t_{k_{1}}}^{t_{k_{1}+1}}\int_{t_{k_{2}}}^{t_{k_{2}+1}}\Vert
X_{1}(t_{k_{1}})X_{2}(t_{k_{1}})-X_{1}(u)X_{2}(u)\Vert _{L^{2}}\Vert
X_{1}(t_{k_{2}})X_{2}(t_{k_{2}})-X_{1}(v)X_{2}(v)\Vert _{L^{2}}dudv \\
& \leqslant \frac{2}{\theta }\frac{1}{{T_{n}}}\sum%
\limits_{k_{1},k_{2}=0}^{n-1}\int_{t_{k_{1}}}^{t_{k_{1}+1}}%
\int_{t_{k_{2}}}^{t_{k_{2}+1}}\Vert X_{i}(t_{k_{1}})-X_{i}(u)\Vert
_{L^{2}}\Vert X_{i}(t_{k_{2}})-X_{i}(v)\Vert _{L^{2}}dudv,\text{ \ }i=1,2,
\end{align*}%
where we used the fact that $X_{1}$ and $X_{2}$ are two Gaussian processes
equal in law. On the other hand since $%
(X_{i}(t_{k})-X_{i}(u))=(Z_{i}(t_{k})-Z_{i}(u))-Z_{0}(e^{-\theta
t_{k}}-e^{-\theta u})$, $i=1,2$. Recall that $%
Z_{i}(r):=\int_{-\infty }^{r}e^{-\theta (r-t)}dW^{i}(t)$, $i=1,2$. We get 
\begin{align*}
\frac{2}{\theta }\frac{1}{{T_{n}}}&
\sum\limits_{k_{1},k_{2}=0}^{n-1}\int_{t_{k_{1}}}^{t_{k_{1}+1}}%
\int_{t_{k_{2}}}^{t_{k_{2}+1}}\Vert X_{1}(t_{k_{1}})-X_{1}(u)\Vert
_{L^{2}}\Vert X_{1}(t_{k_{2}})-X_{1}(v)\Vert _{L^{2}}dudv \\
& \leqslant A_{1}(n)+A_{2}(n)+A_{3}(n)+A_{4}(n)
\end{align*}%
where 
\begin{align*}
& A_{1}(n):=\frac{2}{\theta }\frac{1}{{T_{n}}}\sum%
\limits_{k_{1},k_{2}=0}^{n-1}\int_{t_{k_{1}}}^{t_{k_{1}+1}}%
\int_{t_{k_{2}}}^{t_{k_{2}+1}}\Vert Z_{1}(t_{k_{1}})-Z_{1}(u)\Vert
_{L^{2}}\Vert Z_{1}(t_{k_{2}})-Z_{1}(v)\Vert _{L^{2}}dudv, \\
& A_{2}(n):=\frac{1}{{T_{n}}}\frac{\sqrt{2}}{\theta ^{3/2}}%
\sum\limits_{k_{1},k_{2}=0}^{n-1}\int_{t_{k_{1}}}^{t_{k_{1}+1}}%
\int_{t_{k_{2}}}^{t_{k_{2}+1}}|e^{-\theta t_{k_{2}}}-e^{-\theta v}|\times
\Vert Z_{1}(t_{k_{1}})-Z_{1}(u)\Vert _{L^{2}}dudv, \\
& A_{3}(n):=\frac{1}{{T_{n}}}\frac{\sqrt{2}}{\theta ^{3/2}}%
\sum\limits_{k_{1},k_{2}=0}^{n-1}\int_{t_{k_{1}}}^{t_{k_{1}+1}}%
\int_{t_{k_{2}}}^{t_{k_{2}+1}}|e^{-\theta t_{k_{1}}}-e^{-\theta u}|\times
\Vert Z_{1}(t_{k_{2}})-Z_{1}(v)\Vert _{L^{2}}dudv, \\
& A_{4}(n):=\frac{1}{{T_{n}}}\frac{1}{\theta ^{2}}\sum%
\limits_{k_{1},k_{2}=0}^{n-1}\int_{t_{k_{1}}}^{t_{k_{1}+1}}%
\int_{t_{k_{2}}}^{t_{k_{2}+1}}|e^{-\theta t_{k_{1}}}-e^{-\theta u}|\times
|e^{-\theta t_{k_{2}}}-e^{-\theta v}|dudv.
\end{align*}%
We will use in the sequel the fact that the increments of the process $Z_{1}$
satisfies $\mathbf{E}[(Z_{1}(t)-Z_{1}(s))^{2}]\leqslant |t-s|$, $t,s\geq 0$.
For the first sequence $A_{1}(n)$ we have 
\begin{align*}
A_{1}(n)& \leqslant \frac{2}{\theta }\frac{1}{{T_{n}}}\sum%
\limits_{k_{1},k_{2}=0}^{n-1}\int_{t_{k_{1}}}^{t_{k_{1}+1}}%
\int_{t_{k_{2}}}^{t_{k_{2}+1}}|t_{k_{1}}-u|^{1/2}|t_{k_{2}}-v|^{1/2}dudv, \\
& =\frac{2}{\theta }\frac{\Delta _{n}^{3}}{T_{n}}\sum%
\limits_{k_{1},k_{2}=0}^{n-1}\int_{0}^{1}\int_{0}^{1}t^{1/2}s^{1/2}dtds, \\
& =\frac{8}{9\theta }\times n\Delta _{n}^{2}.
\end{align*}%
where we used the change of variables $t=\frac{u-t_{k_{1}}}{\Delta _{n}}$, $%
s=\frac{v-t_{k_{2}}}{\Delta _{n}}$. For $A_{2}(n)$, we have 
\begin{align*}
A_{2}(n)& \leqslant \frac{\sqrt{2}}{\theta ^{1/2}}\frac{1}{T_{n}}%
\sum\limits_{k_{1},k_{2}=0}^{n-1}\int_{t_{k_{1}}}^{t_{k_{1}+1}}%
\int_{t_{k_{2}}}^{t_{k_{2}+1}}|t_{k_{1}}-u|^{1/2}|t_{k_{2}}-v|dudv, \\
& =\frac{\sqrt{2}}{\theta ^{1/2}}\frac{1}{T_{n}}\Delta
_{n}^{7/2}\sum\limits_{k_{1},k_{2}=0}^{n-1}\int_{0}^{1}%
\int_{0}^{1}t^{1/2}sdtds \\
& \leqslant \frac{\sqrt{2}}{3}\frac{1}{\theta ^{1/2}}\times n\Delta
_{n}^{5/2}.
\end{align*}%
Similarly, we have 
\begin{equation*}
A_{3}(n)\leqslant \frac{\sqrt{2}}{3}\frac{1}{\theta ^{1/2}}\times n\Delta
_{n}^{5/2}.
\end{equation*}%
For the last sequence $A_{4}(n)$, we get 
\begin{align*}
A_{4}(n)& \leqslant \frac{1}{T_{n}}\sum\limits_{k_{1},k_{2}=0}^{n-1}%
\int_{t_{k_{1}}}^{t_{k_{1}+1}}%
\int_{t_{k_{2}}}^{t_{k_{2}+1}}|t_{k_{1}}-u||t_{k_{2}}-v|dudv \\
& \leqslant \frac{\Delta _{n}^{4}}{T_{n}}\sum\limits_{k_{1},k_{2}=0}^{n-1}%
\int_{0}^{1}\int_{0}^{1}tsdtds \\
& \leqslant \frac{n\Delta _{n}^{3}}{4}.
\end{align*}%
The desired result following using the previous inequalities and the fact
that $\Delta _{n}\rightarrow 0$ as $n\rightarrow +\infty $.
\end{proof}

\begin{remark}
One possible mesh that satisfies the assumptions of Lemma \ref{norm-delta}
is $\Delta_n = n^{- \lambda}$ with $1/2 < \lambda <1$.
\end{remark}

\begin{proposition}
\label{An-kol} There exists a constant $C(\theta)$ such that  
\begin{equation*}
d_{Kol}\left(A(n),\mathcal{N}\left(0, \frac{1}{4\theta^3} \right)\right)
\leqslant C(\theta) \times \max\left((n\Delta_n)^{-1/2}, (n\Delta^2_n)^{%
\frac{1}{3}} \right)
\end{equation*}
In particular, if $n\Delta^2_n \rightarrow 0$ as $n \rightarrow +\infty$, we
get  
\begin{equation*}
A(n) \overset{\mathcal{L}}{\longrightarrow} \mathcal{N}\left(0, \frac{1}{%
4\theta^3}\right),
\end{equation*}
as $n \rightarrow +\infty$.
\end{proposition}

\begin{proof}
Since $A(n) = \delta(n) + F_{T_n}$, we have by Proposition \ref{estim-kol}
assertion 1. the following estimate 
\begin{align*}
d_{Kol}\left( A(n), \mathcal{N}(0,\frac{1}{4 \theta^3})\right) & \leqslant
d_{Kol}\left( F_{T_n}, \mathcal{N}(0,\frac{1}{4 \theta^3})\right) + \mathbf{P%
}\left(|2 \theta^{3/2} \delta(n)| > \varepsilon \right) + \varepsilon \\
& \leqslant d_{Kol}\left( F_{T_n}, \mathcal{N}(0,\frac{1}{4 \theta^3}%
)\right) + 4 \theta^3 \times \frac{\mathbf{E}[\delta^2(n)]}{\varepsilon^2} +
\varepsilon \\
& \leqslant \frac{c(\theta)}{\sqrt{T_n}} + \inf\limits_{\varepsilon >
0}g_n(\varepsilon)
\end{align*}
where $g_n(\varepsilon) = 4 \theta^3 \times C_{\theta} \times \frac{n
\Delta^2_n}{\varepsilon^2} + \varepsilon$, since $g_n$ is convex on $\mathbb{%
R}_{+}$, $\arg \inf\limits_{\varepsilon > 0}g_n(\varepsilon) =
\varepsilon^*(n) = \left( 4 \theta^3 \times C_{\theta}
n\Delta^2_n\right)^{1/3}$, where $C_{\theta}$ is the constant from Lemma \ref%
{norm-delta}. Thus, we get the following estimate for the convergence in law
of the random sequence $A(n)$ : 
\begin{equation*}
d_{Kol}\left( A(n), \mathcal{N}(0,\frac{1}{4 \theta^3})\right) \leqslant
C(\theta) \times \max\left((n\Delta_n)^{-1/2}, (n\Delta^2_n)^{\frac{1}{3}}
\right),
\end{equation*}
which ends the proof.
\end{proof}

On the other hand, from the decomposition (\ref{decomY12-discret}),
Corollary \ref{kol-sum}, Proposition \ref{estim-kol}, there exists a
constant $\beta $ such that $4\theta ^{3/2}\frac{\sqrt{3}}{3}\sqrt{T_{n}}%
\mathbf{E}[\tilde{X}_{1}^{2}(n)]<\beta <\left( C\times \mathbf{E}%
[|B(n)|]\wedge 4\right) $, and such that 
\begin{equation}
d_{Kol}\left( \frac{\tilde{Y}_{12}(n)}{\sqrt{T_{n}}},\mathcal{N}(0,\frac{1}{%
4\theta ^{3}})\right) \leqslant d_{Kol}\left( A(n),\mathcal{N}(0,\frac{1}{%
4\theta ^{3}})\right) +\beta \left( 1+\ln (\frac{4}{\beta })\right) ,
\label{kol-tildeY12}
\end{equation}%
Let us now compute $\Vert \tilde{X}_{1}(n)\Vert _{L^{2}}^{2}$, 
\begin{align*}
\tilde{X}_{1}(n)& =\frac{1}{n}\sum\limits_{k=0}^{n-1}I_{1}^{W_{1}}(f_{t_{k}})
\\
& =I_{1}^{W_{1}}(g_{n}),
\end{align*}%
with $g_{n}:=\frac{1}{n}\sum\limits_{k=0}^{n-1}f_{t_{k}}$. Thus similarly to
a previous calculation, we have 
\begin{align}
\mathbf{E}[\tilde{X}_{1}^{2}(n)]& =\Vert g_{n}\Vert _{L^{2}([0,T_{n}])}^{2} 
\notag \\
& =\frac{1}{n^{2}}\sum\limits_{k_{1},k_{2}=0}^{n-1}\langle
f_{t_{k_{1}}},f_{t_{k_{2}}}\rangle _{L^{2}([0,T_{n}])}  \notag \\
& =\frac{1}{n^{2}}\sum\limits_{k_{1},k_{2}=0}^{n-1}\mathbf{E}%
[X_{1}(t_{k_{1}})X_{1}(t_{k_{2}})]  \notag \\
& =\frac{2}{n^{2}}\sum\limits_{k_{1}=0}^{n-1}\mathbf{E}%
[X_{1}(t_{k_{1}})^{2}]+\frac{2}{n^{2}}\sum_{\substack{ k_{1},k_{2}=0 \\ %
k_{1}\neq k_{2}}}^{n-1}\mathbf{E}[X_{1}(t_{k_{1}})X_{1}(t_{k_{2}})]  \notag
\\
& =\frac{2}{n^{2}}\sum\limits_{k_{1}=0}^{n-1}\left( \frac{1-e^{-2\theta
t_{k_{1}}}}{2\theta }\right) +\frac{4}{n^{2}}\sum\limits_{k_{1}=0}^{n-2}\sum%
\limits_{k_{2}=k_{1}+1}^{n-1}\mathbf{E}[X_{1}(t_{k_{1}})X_{1}(t_{k_{2}})] 
\notag \\
& \leqslant \frac{1}{\theta }\frac{1}{n}+\frac{2}{\theta ^{2}}\frac{1}{%
n\Delta _{n}}\leqslant \frac{1}{\theta }\left( 1+\frac{2}{\theta }\right)
(n\Delta _{n})^{-1},  \label{Xtilde-nom}
\end{align}%
where for the last equality we used the fact that 
\begin{equation*}
\frac{4}{n^{2}}\sum\limits_{k_{1}=0}^{n-2}\sum\limits_{k_{2}=k_{1}+1}^{n-1}%
\mathbf{E}[X_{1}(t_{k_{1}})X_{1}(t_{k_{2}})]\leqslant \frac{4}{n^{2}}%
\sum\limits_{k_{1}=0}^{n-2}\sum\limits_{k_{2}=k_{1}+1}^{n-1}\rho
(t_{k_{2}}-t_{k_{1}})
\end{equation*}%
Recall that $\rho (r-s):=\mathbf{E}[Z_{r}^{i}Z_{s}^{i}]$, where $%
Z_{r}^{i}:=\int_{-\infty }^{r}e^{-\theta (r-t)}dW^{i}(t)$, $i=1,2$.
Moreover, 
\begin{align*}
\frac{4}{n^{2}}\sum\limits_{k_{1}=0}^{n-2}\sum\limits_{k_{2}=k_{1}+1}^{n-1}%
\rho (t_{k_{2}}-t_{k_{1}})& =\frac{4}{n^{2}}\sum\limits_{k_{1}=0}^{n-2}\sum%
\limits_{k_{2}=k_{1}+1}^{n-1}\rho (\Delta _{n}(k_{2}-k_{1})) \\
& =\frac{4}{n^{2}}\sum\limits_{r=1}^{n-1}(n-r)\rho (\Delta _{n}r) \\
& =\frac{2}{\theta }\frac{1}{n^{2}}\sum\limits_{r=1}^{n-1}(n-r)e^{-\theta
\Delta _{n}r} \\
& \leqslant \frac{2}{\theta }\frac{1}{n}\frac{1}{(1-e^{-\Delta _{n}\theta })}
\\
& \leqslant \frac{2}{\theta ^{2}}(n\Delta _{n})^{-1}.
\end{align*}%
We deduce that there exists a constant $c(\theta )$, such that 
\begin{align*}
\mathbf{E}[|B(n)|]& =\sqrt{T_{n}}\mathbf{E}[|\tilde{X}_{1}(n)|]\mathbf{E}[|%
\tilde{X}_{2}(n)|] \\
& \leqslant c(\theta )(n\Delta _{n})^{-1/2}.
\end{align*}%
Therefore, from equation (\ref{kol-tildeY12}) and the fact that the function 
$x\mapsto x\left( 1+\ln (\frac{4}{x})\right) $ on $(0,4)$, we have for $n$
large enough 
\begin{equation*}
d_{Kol}\left( \frac{\tilde{Y}_{12}(n)}{\sqrt{T_{n}}},\mathcal{N}(0,\frac{1}{%
4\theta ^{3}})\right) \leqslant d_{Kol}\left( A(n),\mathcal{N}(0,\frac{1}{%
4\theta ^{3}})\right) +\frac{c(\theta )}{2}\frac{\ln (n\Delta _{n})}{\sqrt{%
n\Delta _{n}}}.
\end{equation*}%
The following proposition follows.

\begin{theorem}
There exists a constant $c(\theta)$ depending only on $\theta$ such that 
\begin{equation*}
d_{Kol} \left( \frac{\tilde{Y}_{12}(n)}{\sqrt{T_n}}, \mathcal{N}(0, \frac{1}{%
4 \theta^3})\right) \leqslant c(\theta) \times \ln(n\Delta_n)
\max\left((n\Delta_n)^{-1/2}, (n\Delta^2_n)^{\frac{1}{3}} \right)
\end{equation*}
In particular, if $n \Delta^2_n \rightarrow 0 $ as $n \rightarrow +\infty$, 
\begin{equation*}
\frac{\tilde{Y}_{12}(n)}{\sqrt{T_n}} \overset{\mathcal{L}}{\longrightarrow} 
\mathcal{N}\left(0, \frac{1}{4\theta^3}\right),
\end{equation*}
as $n \rightarrow +\infty$.
\end{theorem}

\begin{example}
\label{boundsY12}  If $\Delta_n = n^{-\lambda}$ with $\frac{1}{2} < \lambda
<1$, then we have  
\begin{equation*}
d_{Kol}\left( \frac{\tilde{Y}_{12}(n)}{\sqrt{T_n}}, \mathcal{N}(0, \frac{1}{%
4 \theta^3})\right) \leqslant C(\theta) \times(1-\lambda)\times\left\{ 
\begin{array}{ll}
\ln(n)\times n^{\frac{1- 2\lambda}{3}} & \mbox{ if } \frac{1}{2} < \lambda
\leqslant \frac{5}{7} \\ 
&  \\ 
\ln(n) \times {n ^{\frac{\lambda-1}{2}}} & \mbox{ if } \frac{5}{7} \leqslant
\lambda < 1.%
\end{array}
\right.
\end{equation*}
Consequently,  
\begin{equation*}
\frac{\tilde{Y}_{12}(n)}{\sqrt{T_n}} \overset{\mathcal{L}}{\longrightarrow} 
\mathcal{N}\left(0, \frac{1}{4\theta^3}\right)
\end{equation*}
as $n \rightarrow +\infty$.
\end{example}

\subsection{The denominator term}

We denote the denominator term in $\tilde{\rho}(n)$ by $\tilde{D}(n)$, i.e.  
\begin{equation}
\tilde{D}(n):=2\theta \sqrt{\frac{\tilde{Y}_{11}(n)}{T_{n}}\times \frac{%
\tilde{Y}_{22}(n)}{T_{n}}}  \label{denom-discret}
\end{equation}%
According to Proposition \ref{estim-kol} assertion 2, we need to estimate $%
\mathbf{P}(|\tilde{D}(n)-1|>\varepsilon )$ for instance for $0<\varepsilon <1
$. Thus, denoting $\bar{Y}_{11}(n):=\frac{\tilde{Y}_{11}(n)}{T_{n}/2\theta }$%
, then similarly to the calculations performed in the continuous case, i.e. (%
\ref{estim-Z-continu}), we get 
\begin{equation*}
\mathbf{P}\left( |\tilde{D}(n)-1|>\varepsilon \right) \leqslant 4\mathbf{P}%
\left( |\bar{Y}_{11}(n)-1|>\frac{\varepsilon }{3}\right) 
\end{equation*}%
Writing $\varepsilon $ instead of $\varepsilon /3$, by Proposition \ref%
{estim-kol} assertion 2 applied to $\tilde{\rho}(n)$, we get 
\begin{equation}
d_{Kol}\left( \sqrt{\theta }\sqrt{T_{n}}\tilde{\rho}(n),N\right) \leqslant
d_{Kol}\left( 2\theta ^{3/2}\frac{\tilde{Y_{12}}(n)}{\sqrt{T_{n}}},N\right)
+4\times \mathbf{P}\left( |\bar{Y}_{11}(n)-1|>\varepsilon \right)
+3\varepsilon .  \label{rhoestim-discret}
\end{equation}%
where $N\sim \mathcal{N}(0,1)$. It remains to control the term $\mathbf{P}%
\left( |\bar{Y}_{11}(n)-1|>\varepsilon \right) $. we have 
\begin{equation}
\bar{Y}_{11}(n)=D_{1,n}+D_{2,n}-2\theta \tilde{X}_{1}^{2}(n),
\label{decomp-Y11}
\end{equation}%
where 
\begin{equation}
D_{1,n}:=\frac{2\theta }{n}\sum\limits_{k=0}^{n-1}\left( X_{1}^{2}(t_{k})-%
\mathbf{E}[X_{1}^{2}(t_{k})]\right) \quad \text{and}\quad D_{2,n}:=\frac{%
2\theta }{n}\sum\limits_{k=0}^{n-1}\mathbf{E}[X_{1}^{2}(t_{k})]  \label{D12}
\end{equation}

\begin{lemma}
\label{D1n-norm} Consider $D_{1,n}$ defined in (\ref{D12}), then for every
large $n $, we have 
\begin{equation*}
\mathbf{E}\left[D^2_{1,n}\right] \leqslant 2\left(1+ \frac{1}{\theta}%
\right)\times (n \Delta_{n})^{-1}.
\end{equation*}
\begin{proof}
The sequence $D_{1,n}$ can be written as follows :  
\begin{align*}
D_{1,n} & = \frac{2 \theta}{n} \sum \limits_{k=0}^{n-1} \left(
I_{1}^{W_1}(f_{t_{k}}) - \|f_{t_{k}}\|_{L^{2}([0,T_n])} \right) \\
& = \frac{2 \theta}{n} \sum\limits_{k=0}^{n-1} I^{W_1}_{2}(f^{\otimes
2}_{t_{k}}) \\
& = I^{W_1}_{2}(k_n),
\end{align*}
where $k_n := \frac{2 \theta}{n}\sum\limits_{k=0}^{n-1}f^{\otimes 2}_{t_{k}}$%
. Thus  
\begin{align*}
\mathbf{E}[D^{2}_{1,n}] & = 2 \|k_n\|^2_{L^{2}([0,T_n]^2)} \\
& = \frac{2 \times (2\theta)^{2}}{n^{2}} \sum \limits_{k_1,k_2 =0}^{n-1}
\langle f^{\otimes 2}_{t_{k_1}}, f^{\otimes
2}_{t_{k_2}}\rangle_{L^{2}([0,T_n]^2)} \\
& = \frac{2 \times (2\theta)^{2}}{n^{2}} \sum \limits_{k_1,k_2 =0}^{n-1}
\left(\langle f_{t_{k_1}}, f_{t_{k_2}}\rangle_{L^{2}([0,T_n])} \right)^2 \\
& = \frac{2 \times (2\theta)^{2}}{n^{2}} \sum \limits_{k_1,k_2 =0}^{n-1}
\left( \mathbf{E}[X_1(t_{k_1})X_1(t_{k_2})] \right)^2 \\
&= \frac{2 \times (2\theta)^{2}}{n^2} \sum\limits_{k_1 =0}^{n-1} \mathbf{E}%
[X_1(t_{k_1})^2]^2 + \frac{2 \times (2\theta)^{2}}{n^{2}} \sum \sum
_{\substack{ k_{1},k_{2}= 0 \\ k_{1} \neq k_{2} }}^{n-1} \left( \mathbf{E}%
[X_1(t_{k_1})X_1(t_{k_2})] \right)^2 \\
& = \frac{2 \times (2\theta)^{2}}{n^{2}} \sum \limits_{k_1=0}^{n-1} \left( 
\frac{1-e^{-2 \theta t_{k_1}}}{2 \theta}\right)^{2} + \frac{4 \times
(2\theta)^{2}}{n^{2}} \sum\limits_{k_1=0}^{n-2} \sum \limits_{k_2 =
k_1+1}^{n-1} \left( \mathbf{E}[X_1(t_{k_1})X_1(t_{k_2})] \right)^2 \\
& \leqslant \frac{2}{n} + \frac{4 \times (2\theta)^{2}}{n^{2}}
\sum\limits_{k_1=0}^{n-2} \sum \limits_{k_2 = k_1+1}^{n-1} \left( \mathbf{E}%
[X_1(t_{k_1})X_1(t_{k_2})] \right)^2.
\end{align*}
For the right hand partial sum, we can write  
\begin{equation*}
\frac{4}{n^{2}} \sum\limits_{k_1=0}^{n-2} \sum \limits_{k_2 = k_1+1}^{n-1}
\left( \mathbf{E}[X_1(t_{k_1})X_1(t_{k_2})] \right)^2 \leqslant \frac{4}{%
n^{2}} \sum\limits_{k_1=0}^{n-2} \sum \limits_{k_2 = k_1+1}^{n-1}
\rho^2(t_{k_{2}}-t_{k_{1}})
\end{equation*}
Moreover,  
\begin{align*}
\frac{4}{n^{2}} \sum\limits_{k_1=0}^{n-2} \sum \limits_{k_2 = k_1+1}^{n-1}
\rho^2(t_{k_{2}}-t_{k_{1}}) & = \frac{4}{n^2}\sum\limits_{k_1=0}^{n-2} \sum
\limits_{k_2 = k_1+1}^{n-1} \rho^2(\Delta_n(k_{2}- k_{1})) \\
& = \frac{4}{n^2} \sum \limits_{r=1}^{n-1}(n-r) \rho^{2}(\Delta_n r) \\
& = \frac{1}{\theta^2} \frac{1}{n^2} \sum \limits_{r=1}^{n-1}(n-r) e^{-2
\theta \Delta_n r} \\
& \leqslant \frac{1}{\theta^2} \frac{1}{n} \frac{1}{(1-e^{-2 \Delta_n
\theta})} \\
& \leqslant \frac{1}{2 \theta^3} \frac{1}{n \Delta_n}.
\end{align*}
For every large $n$. The desired result follows since $\max\left(n^{-1},
(n\Delta_n)^{-1}\right) = (n\Delta_n)^{-1} $.
\end{proof}
\end{lemma}
For the sequence $D_{2,n}$, we have 
\begin{align}
|D_{2,n}-1|& \leqslant \frac{2\theta }{n}\sum\limits_{k=0}^{n-1}\left\vert 
\mathbf{E}[X_{1}^{2}(t_{k})]-\frac{1}{2\theta }\right\vert   \notag \\
& \leqslant \frac{2\theta }{n}\sum\limits_{k=0}^{n-1}e^{-2\theta k\Delta
_{n}}\leqslant \frac{2\theta }{n}\frac{1}{(1-e^{-2\theta \Delta _{n}})}%
\leqslant \frac{1}{n\Delta _{n}}  \label{D2n-norm}
\end{align}%
For every large $n$, $D_{2,n}=1+O\left( (n\Delta _{n})^{-1}\right) $.
Therefore, following exactly the same analysis done for the denominator term
in the continuous case and choosing $\varepsilon =\varepsilon (n)=c\times 
\frac{\ln (n\Delta _{n})}{\sqrt{n\Delta _{n}}}$, we get the existence of a
constant $c(\theta )$, such that for $n$ large enough 
\begin{equation}
\mathbf{P}\left( |\tilde{D}(n)-1|>\varepsilon (n)\right) \leqslant c(\theta )%
\frac{\ln (n\Delta _{n})}{\sqrt{n\Delta _{n}}},
\end{equation}%
The previous estimates will allow to prove first a Strong law result by
showing that $\tilde{\rho}(n)$ converges to $0$ almost surely as $%
n\rightarrow +\infty $, then the convergence in law of the statistic $\tilde{%
\rho}(n)$ given with its rate of convergence as $n\rightarrow +\infty $.

\begin{proposition}
\label{lln} Assume that $\Delta_n = n^{-\lambda}$, for $\frac{1}{2}< \lambda
< 1$, then we have almost surely 
\begin{equation*}
\tilde{\rho}(n) \longrightarrow 0 \quad \text{as} \quad n \rightarrow
+\infty.
\end{equation*}
\end{proposition}
\begin{proof}
We can write $\tilde{\rho}(n)$ as follows 
\begin{equation*}
\tilde{\rho}(n)=\frac{\frac{\tilde{Y}_{12}(n)}{T_{n}}}{\sqrt{\frac{\tilde{Y}%
_{11}(n)}{T_{n}}\times \frac{\tilde{Y}_{22}(n)}{T_{n}}}}
\end{equation*}%
For the numerator term, we have by the decomposition (\ref{decomY12-discret}%
), 
\begin{equation*}
\frac{\tilde{Y}_{12}(n)}{T_{n}}=\frac{A(n)}{\sqrt{T_{n}}}-\frac{B(n)}{\sqrt{%
T_{n}}}
\end{equation*}%
For the sequence $A(n)$, we have 
\begin{equation*}
\frac{A(n)}{\sqrt{T_{n}}}=\frac{\delta (n)}{\sqrt{T_{n}}}+\frac{F_{T_{n}}}{%
\sqrt{T_{n}}},
\end{equation*}%
Then by Lemma \ref{norm-delta}, we get $\mathbf{E}\left[ \left( \frac{\delta
(n)}{\sqrt{T_{n}}}\right) ^{2}\right] \leqslant C_{\theta }n^{-\lambda }$,
and by the hypercontractivity property and Lemma \ref{Borel-Cantelli}, we
obtain $\frac{\delta (n)}{\sqrt{T_{n}}}\rightarrow 0$, a.s. as $n\rightarrow
+\infty $. By the same argument, and using Lemma \ref{varianceFT}, we obtain 
$\mathbf{E}\left[ \left( \frac{F_{T_{n}}}{\sqrt{T_{n}}}\right) ^{2}\right]
\leqslant C_{\theta }n^{-(1-\lambda )}$, and we then get $\frac{F_{T_{n}}}{%
\sqrt{T_{n}}}\rightarrow 0$, a.s. as $n\rightarrow +\infty $. Consequently,
we have a.s. as $n\rightarrow +\infty $, $\frac{A(n)}{\sqrt{T_{n}}}%
\rightarrow 0.$ 
For the sequence $B(n)$, $\mathbf{E}\left[ \left( \frac{B(n)}{\sqrt{T_{n}}}%
\right) ^{2}\right] =\frac{1}{T_{n}}\mathbf{E}[\tilde{X}_{1}^{2}(n)]\mathbf{E%
}[\tilde{X}_{2}^{2}(n)]\leqslant C\times n^{-(1-\lambda )}$, thus $\frac{B(n)%
}{\sqrt{T_{n}}}$ a.s. as $n\rightarrow +\infty $. Finally, 
\begin{equation*}
\frac{\tilde{Y}_{12}(n)}{{T_{n}}}\longrightarrow 0,
\end{equation*}%
a.s. as $n\rightarrow +\infty .$
For the denominator term, we will need the following proposition
\begin{proposition}
\label{denom-discret}  For every $p \geq 1$, there exists a constant ${c}%
(p,\theta)$ depending on $p$ and $\theta$, such that  
\begin{equation*}
\mathbf{E}\left[ \left|2 \theta \sqrt{\frac{\tilde{Y}_{11}(n)}{T_n}\times 
\frac{\tilde{Y}_{22}(n)}{T_n}} -1 \right|^{p}\right]^{\frac{1}{p}} \leqslant 
{c}(p,\theta) \times T^{-\frac{1}{2}}_{n}
\end{equation*}
\end{proposition}
\begin{proof}
Using the fact that if $X\geq 0$ p.s. then we have $\mathbf{E}[|\sqrt{X}%
-1|^{p}]\leqslant \mathbf{E}[|X-1|^{p}]$ for every $p>0$, then by the
notation $\bar{Y}_{ii}(n)=\frac{\tilde{Y}_{ii}(n)}{T_{n}/2\theta }$, $i=1,2$%
, we get 
\begin{align}
\mathbf{E}\left[ \left\vert 2\theta \sqrt{\frac{\tilde{Y}_{11}(n)}{T_{n}}%
\times \frac{\tilde{Y}_{22}(n)}{T_{n}}}-1\right\vert ^{p}\right] ^{1/p}&
\leqslant \times \mathbf{E}\left[ \left\vert \bar{Y}_{11}(n)\times \bar{Y}%
_{22}(n)-1\right\vert ^{p}\right] ^{1/p}  \notag  \label{denomterm-dis} \\
& \leqslant \mathbf{E}\left[ |\bar{Y}_{11}(n)|^{p}\right] ^{1/p}\times 
\mathbf{E}\left[ |\bar{Y}_{22}(n)-1|^{p}\right] ^{1/p}+\mathbf{E}\left[ |%
\bar{Y}_{11}(n)-1|^{p}\right] ^{1/p}  \notag
\end{align}%
Moreover, since by the decomposition (\ref{decomp-Y11}), $\bar{Y}%
_{11}(n)=D_{1,n}+D_{2,n}-2\theta \bar{X}_{1}^{2}(n)$, where $D_{1,n}$ and $%
D_{2,n}$ are defined in (\ref{D12}), we get 
\begin{align*}
\mathbf{E}\left[ |\bar{Y}_{11}(n)-1|^{p}\right] ^{1/p}& \leqslant \mathbf{E}%
\left[ |D_{1,n}|^{p}\right] ^{1/p}+|D_{2,n}-1|+2\theta \mathbf{E}\left[ 
\tilde{X}_{1}^{2}(n)^{p}\right] ^{1/p} \\
& \leqslant c(p,\theta )\times T_{n}^{-1/2},
\end{align*}%
where we used the hypercontractivity property (\ref{hypercontractivity}),
Lemma \ref{D1n-norm} and the estimate (\ref{D2n-norm}), with 
\begin{equation*}
c(p,\theta )=3\times \max \left( \left\{ (p-1)\mathbf{1}_{p\geq 2}+\mathbf{1}%
_{p=1}\right\} \sqrt{{2}\left( 1+\frac{1}{\theta }\right) },1,(2p-1)\frac{1}{%
\theta }\left( 1+\frac{2}{\theta }\right) \right) .
\end{equation*}%
The result of Proposition \ref{denom-discret} is therefore established.
\end{proof}
By Proposition \ref{denom-discret}, we have for all $\eta >0$, 
\begin{align*}
\sum\limits_{n=1}^{+\infty }\mathbf{P}\left( \left\vert 2\theta \sqrt{\frac{%
\tilde{Y}_{11}(n)}{T_{n}}\times \frac{\tilde{Y}_{22}(n)}{T_{n}}}%
-1\right\vert >\eta \right) & \leqslant \frac{1}{\eta ^{p}}%
\sum\limits_{n=1}^{+\infty }\mathbf{E}\left[ \left\vert 2\theta \sqrt{\frac{%
\tilde{Y}_{11}(n)}{T_{n}}\times \frac{\tilde{Y}_{22}(n)}{T_{n}}}%
-1\right\vert ^{p}\right]  \\
& \leqslant \frac{1}{\eta ^{p}}\sum\limits_{n=1}^{+\infty }\frac{1}{%
n^{(1-\lambda )p/2}}<+\infty ,
\end{align*}%
for any $p>\frac{2}{1-\lambda }$, it follows from Borel-Cantelli's Lemma
that 
\begin{equation*}
\sqrt{\frac{\tilde{Y}_{11}(n)}{T_{n}}\times \frac{\tilde{Y}_{22}(n)}{T_{n}}}%
\rightarrow \frac{1}{2\theta },
\end{equation*}%
almost surely as $n\rightarrow +\infty $, which finishes the proof.
\end{proof}

\begin{remark}
The previous proposition was proved in the scale $\Delta _{n}=n^{-\lambda }$%
, for $\frac{1}{2}<\lambda <1$, for ease of presentation, but one can also
show that it holds for any mesh $\Delta _{n}$ satisfying:

\begin{itemize}
\item $n \Delta^{\alpha}_n \rightarrow 0$, a.s. for some $\alpha > 1$ as $n
\rightarrow +\infty$.

\item $\sum\limits_{n \geq 1} \frac{1}{n^{p} \Delta^{p}_n} < +\infty$ for
any $p > 1$.
\end{itemize}
\end{remark}

\begin{theorem}
\label{rho-case-discrete} There exists a constant $c(\theta)$ such that for $%
n$ large enough, we have 
\begin{equation*}
d_{Kol} \left( \sqrt{\theta} \sqrt{T_n} \tilde{\rho}(n), \mathcal{N}(0,
1)\right) \leqslant c(\theta) \times \ln(n\Delta_n)
\max\left((n\Delta_n)^{-1/2}, (n\Delta^2_n)^{\frac{1}{3}} \right)
\end{equation*}
In particular, if $n \Delta^2_n \rightarrow 0 $ as $n \rightarrow +\infty$, 
\begin{equation*}
\sqrt{T_n} \tilde{\rho}(n) \overset{\mathcal{L}}{\longrightarrow} \mathcal{N}%
\left(0, \frac{1}{\theta}\right),
\end{equation*}
as $n \rightarrow +\infty$.
\end{theorem}

\begin{remark}
\label{continu-vs-discret} The results obtained in Theorem \ref%
{rho-case-discrete} can be as efficient as those of Theorem  \ref%
{rho-case-continuous}, as long as one picks a step size very precisely. In
fact, Theorem \ref{rho-case-discrete} allows us to identify an optimal step
size $\Delta _{n}$ immediately, by requiring that $\left( n\Delta
_{n}^{2}\right) ^{1/3}$ is of the same order as $\left( n\Delta _{n}\right)
^{-1/2}$. By equating these two terms, we immediately find that it is
optimal to choose $\Delta _{n}$ on the order of $n^{-5/7}$. When choosing $%
\Delta _{n}=n^{-5/7}$, one then immediately finds that $T_{n}=n\Delta
_{n}=n^{2/7}$, which means that the speed in the Kolmogorov metric in
Theorem \ref{rho-case-discrete} is bounded above by $\ln (n)n^{-1/7}$ up to
a constant. Therefore, in terms of $T_{n}$, the rate of convergence is of
the order of $\ln (T_{n})\times T_{n}^{-1/2}$. We obtain therefore the same
speed as in the continuous case in Theorem \ref{rho-case-continuous}. It is
in this sense that the convergence rate in Theorem \ref{rho-case-discrete}
is as efficient in Theorem \ref{rho-case-continuous}.
\end{remark}

\section{Numerical results}

This section contains a numerical study of some of the properties of the
discrete version of Yule's nonsense correlation statistic, which we denoted
by $\tilde{\rho}(n)$ in (\ref{ro-discrete}). We first simulate the OU
processes $X_{1}$ and $X_{2}$ according to the following steps :

\begin{enumerate}
\item Set the values of $\theta$, the sample size $n$ and the mesh $\Delta_n
= n^{- \lambda}$, $\frac{1}{2} < \lambda < 1$.

\item Generate two independent Brownian motions $W^1$ and $W^2$.

\item Set $X^i_0 =0$, for $i=1,2$ and simulate the observations $%
X^i_{\Delta_n}$, $X^i_{2\Delta_n}$, ..., $X^i_{T_n}$, $i=1,2$ where $T_n =
n\times \Delta_n$ following the Euler scheme : 
\begin{equation*}
X^{i}_{t_j} = (1 - \theta \Delta_n)X^{i}_{t_{j-1}} + \left(W^{i}_{t_j} -
W^{i}_{t_{j-1}} \right) \quad j=1,..,n, \quad i=1,2.
\end{equation*}

\item We obtain a simulation of the sample paths of $X^1$ and $X^2$ based on 
$X^{i}_{t_{j}}$, $j=1,..,n$, $i=1,2$ by approximating $X^1$ and $X^2$ using
the linear process linking the points $(t_{j}, X^{i}_{t_{j}})_{\{ 1
\leqslant j \leqslant n\}}$ for $i=1,2$ as follows 
\begin{equation*}
X^i_{t} = \left(1- \theta(t-t_{j-1})X^{i}_{t_{j-1}} \right) +
\left(W^{i}_{t} - W^{i}_{t_{j-1}} \right), \quad \text{for all} \text{ } t
\in [t_j, t_{j-1}], \text{ \ } j=1,...,n.
\end{equation*}
The figures below are an example of four sample paths of $X^1$ and $X^2$ for
different values of the drift parameter $\theta$. 
\begin{figure}[h]
\begin{subfigure}[b]{0.5\linewidth}
		\centering
		\includegraphics[width=1\linewidth]{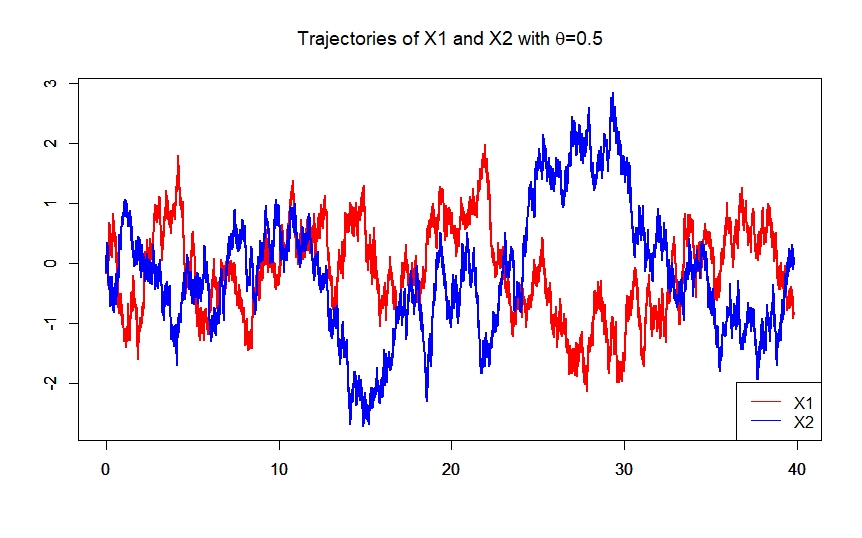}
	\end{subfigure}
\begin{subfigure}[b]{0.5\linewidth}
		\centering
		\includegraphics[width=1\linewidth]{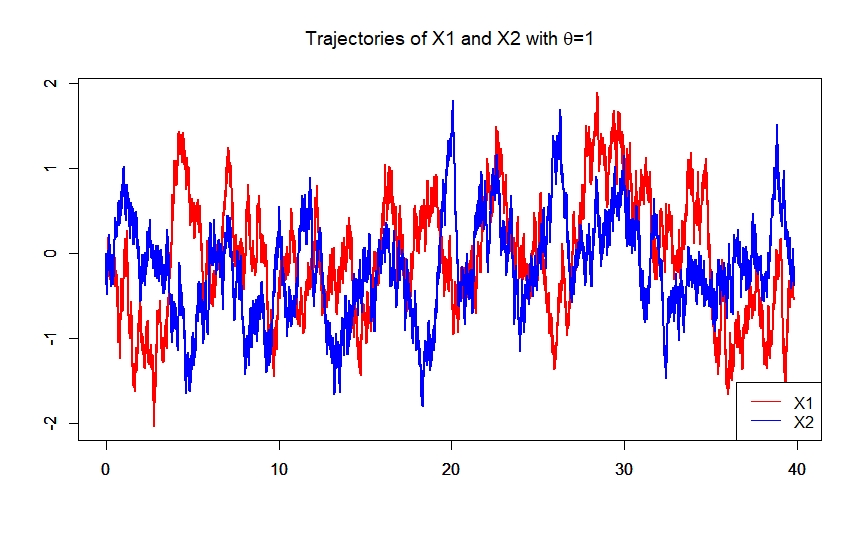}
	\end{subfigure}
\begin{subfigure}[b]{0.5\linewidth}
		\centering
		\includegraphics[width=1\linewidth]{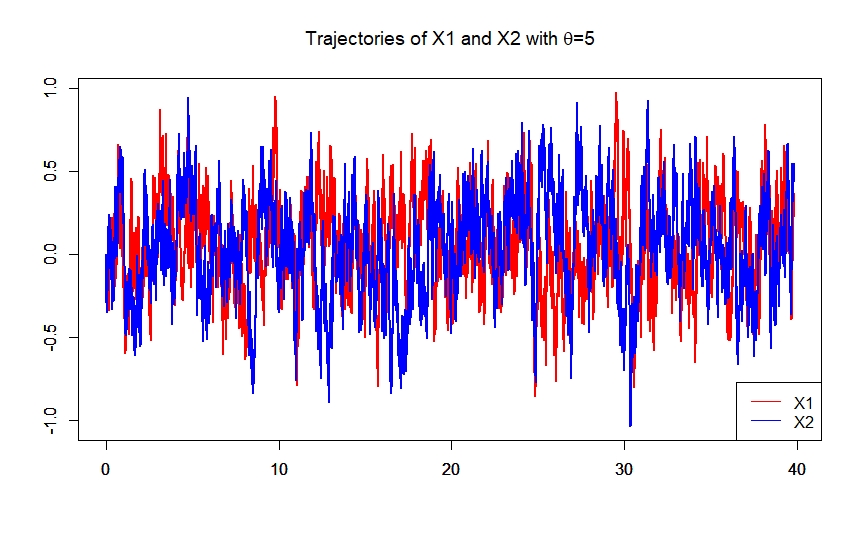}
	\end{subfigure}
\begin{subfigure}[b]{0.5\linewidth}
		\centering
		\includegraphics[width=1\linewidth]{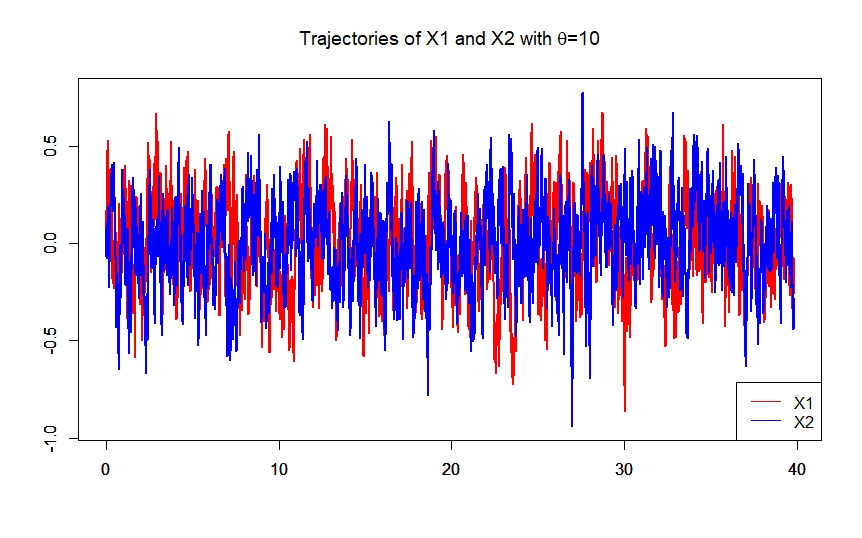}
	\end{subfigure}
\caption{Sample paths of X1 and X2 for different values of $\protect\theta$.}
\end{figure}
\end{enumerate}
The simulation of the Ornstein-Uhlenbeck sample paths $X_{1}$ and $X_{2}$ is
done for different values of the parameter $\theta =\{0.5,1,5,10\}$, for a
sample size $n=10000$ and a mesh $\Delta _{n}=10000^{-0.6}$ ($\lambda =0.6$%
) which corresponds to a time horizon $T_{n}\sim 40$. One can see from the
figures above how the drift parameter value impacts the variability and
raggedness of OU sample paths.
The next step is to illustrate numerically Proposition \ref{lln}. The table
below shows the mean, the median and standard deviation values for $\tilde{%
\rho}(n)$ for different values of $n$, using 500 Monte-Carlo replications
for three different values of the drift parameter $\theta =\{1,5,10\}$. 
\begin{table}[h]
\centering
\begin{tabular}{lcccccccccccc}
\hline
&  &  &  &  &  &  &  &  &  &  &  &  \\ 
&  & \textbf{$n= 10000$} &  & {$n= 50000$} &  & {$n= 100000$} &  &  &  &  & 
&  \\ 
&  & {$T_{n} \sim 40$} &  & {$T_{n} \sim 76$} &  & {$T_{n} \sim 100$} &  & 
&  &  &  &  \\ \hline
\multirow{3}{13 mm}{\bf $\theta = 1$} & Mean & -0.01022 &  & 0.00667 &  & 
0.00377 &  &  &  &  &  &  \\ 
& Median & -0.00725 &  & 0.00873 &  & 0.00168 &  &  &  &  &  &  \\ 
& S.Dev & 0.14990 &  & 0.10162 &  & 0.10898 &  &  &  &  &  &  \\ \hline
\multirow{3}{13 mm}{\bf $\theta = 5$} & Mean & 0.00296 &  & 0.00130 &  & 
0.00088 &  &  &  &  &  &  \\ 
& Median & 0.00136 &  & 0.00214 &  & -0.0011 &  &  &  &  &  &  \\ 
& S.Dev & 0.05237 &  & 0.06892 &  & 0.04602 &  &  &  &  &  &  \\ \hline
\multirow{3}{13 mm}{\bf $\theta = 10$} & Mean & -0.00258 &  & -0.00038 &  & 
0.00015 &  &  &  &  &  &  \\ 
& Median & -0.00147 &  & 0.00036 &  & 0.00035 &  &  &  &  &  &  \\ 
& S.Dev & 0.04935 &  & 0.03641 &  & 0.03156 &  &  &  &  &  &  \\ \hline
\end{tabular}
\caption{Estimation results for $n=\{10000,50000,100000\}$ and $\protect%
\lambda =0.6$.}
\label{1}
\end{table}
Table \ref{1} above shows that $\tilde{\rho}(n)$ approaches zero for large
values of the sample size $n$ which confirms Proposition \ref{lln}, even for
moderate $\theta $, and even though $T=40$ is not an inordinately large
value.
To investigate the asymptotic normal distribution of $\tilde{\rho}(n)$
empirically, we need to compare the distribution of the statistic 
\begin{equation*}
\psi (n,\theta ):=\sqrt{T_{n}}\sqrt{\theta }\tilde{\rho}(n)
\end{equation*}%
with the standard Gaussian distribution $\mathcal{N}(0,1)$. For this aim, we
chose $\theta =2$, $n=100000$, $T_{n}=100$ and based on 3000 replications,
we obtained the following histogram : 
\begin{figure}[h]
\centering
\includegraphics[width=0.7\linewidth]{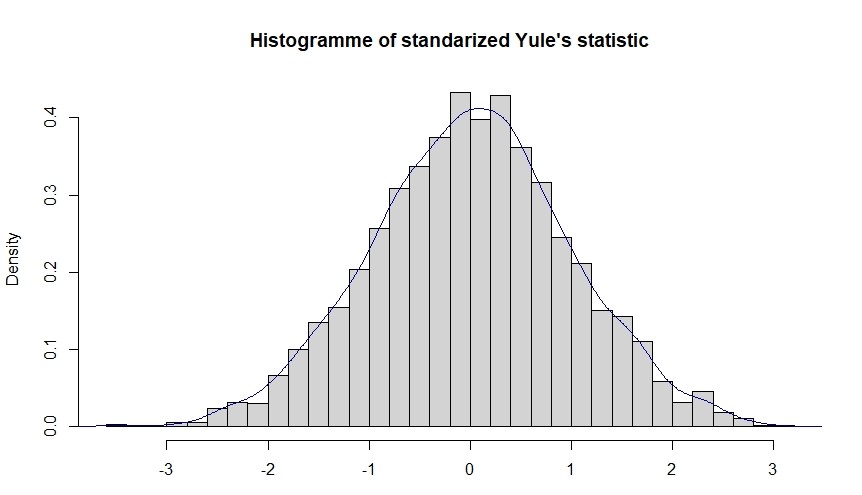}  
\caption{Histogram of $\protect\psi (n,\protect\theta )$ for $n=100000$, $%
T_{n}=100$, $\protect\theta =2$}
\label{fig:ar1}
\end{figure}

The histogram (\ref{fig:ar1}) shows visually that the normal approximation
of the distribution of the statistic $\psi (n,\theta )$ is reasonable for
the sampling size $n=100000$ and time horizon $T_{n}=100$. Moreover, the
results of the next table below show a comparison of statistical properties
between $\psi (n,\theta )$ and $\mathcal{N}$(0,1) with the same parameters
as in Figure \ref{fig:ar1}, we can see that the empirical mean, median and
standard deviation of $\psi (n,\theta )$ and $\mathcal{N}$(0,1) are quite
close, which illustrates well our theoretical results.

\begin{center}
\begin{tabular}{cccc}
\multicolumn{4}{c}{} \\ 
Statistics & Mean & Median & Standard Deviation \\ \hline
$\mathcal{N}$(0,1) & 0 & 0 & 1 \\ 
$\psi(n,\theta)$ & 0.00832 & 0.01206 & 0.99691 \\ \hline
\end{tabular}
\end{center}
We can also illustrate numerically the rate of convergence in law of the
statistic $\psi (n,\theta )$ to the standard Gaussian distribution, by
approximately computing the Kolmogorov distance between $\psi (n,\theta )$
and $\mathcal{N}(0,1)$. For this aim, we approximate the cumulative
distribution function using an empirical cumulative distribution function
based on 500 replications of the computation of $\psi (n,\theta )$ for $%
n=100000$. The figure below shows the empirical and standard normal
cumulative distribution functions. 
\begin{figure}[h]
\centering
\includegraphics[width=0.8\linewidth]{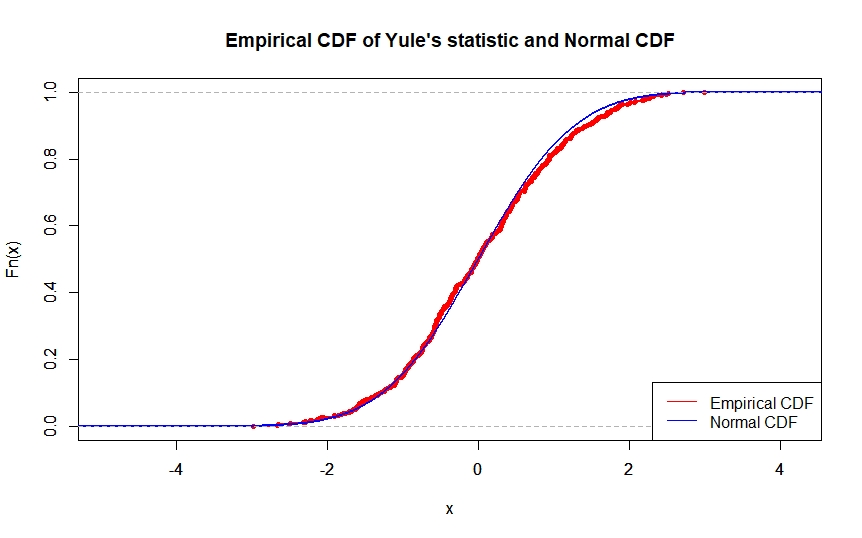}  \label{fig:cdf}
\end{figure}

Based on Remark \ref{continu-vs-discret}, when the mesh $\Delta
_{n}=n^{-\lambda }$ for $\frac{1}{2}<\lambda <\frac{5}{7}$, we expect that 
\begin{equation*}
d_{Kol}\left( \sqrt{T_{n}}\sqrt{\theta }\tilde{\rho}(n),\mathcal{N}%
(0,1)\right) \leqslant c(\theta )(1-\lambda )\ln (n)\times n^{\frac{%
1-2\lambda }{3}}.
\end{equation*}%
In fact, with our choice of $\lambda =0.6$ and a sample size $n=100000$,
the time horizon $T_{n}=100$. The mesh size $n^{-0.6}$ is larger than the
optimal size $n^{-5/7},$ yielding a larger time horizon $T_{n}$ than under
the optimal observation frequency. The Kolmogorov distance between the two
laws, which equals the sup norm of the difference of these cumulative
distribution functions, computes to approximately 0.01974, which implies
that $c(\theta )$ is greater than $9.310^{-3}$. We could have chosen the
optimal mesh{}%
\begin{equation*}
\Delta _{n}=n^{-5/7},
\end{equation*}%
yielding a rate of order $\ln (n)\times n^{-1/7}$, but in this case in order
to have the same time horizon $T_{n}=100$, we would have needed $n=10^{7}$
data points, which is a large number. In practical applications, the cost of
higher-frequency observations, if known, is to be balanced with desired
precision on the Kolmogorov distance, which may well point to a lower
frequency for a fixed time horizon.

\end{document}